\documentclass[11pt]{amsart}

\usepackage{
  amsmath,
  amsfonts,
  amssymb,
  amsthm,
  amscd,
  graphicx,
  comment,      % To comment out a section
  etoolbox,     % For toggles
  mathtools,    % Fix some bugs in amsmath, add some functionality
  enumitem,
  mathdots,     % Fixes some problems with \dots, \dots, etc.
  stmaryrd,     % St. Mary's Road symbols font
  booktabs,     % For nicer tables
  stackrel,     % For putting something below the relation
}
\usepackage[usenames,dvipsnames]{xcolor}
\usepackage[colorlinks=true, linkcolor=blue, citecolor=blue, urlcolor=blue, breaklinks=true]{hyperref}

\usepackage[capitalize]{cleveref}   % Can add 'nameinlink' option to have the name included in hyperlink

\crefname{defin}{Definition}{Definitions}
\crefname{eg}{Example}{Examples}
\crefname{lem}{Lemma}{Lemmas}
\crefname{theo}{Theorem}{Theorems}
\crefname{equation}{}{}
\crefname{enumi}{}{}
\newcommand\C{\mathbb{C}}
\newcommand\Z{\mathbb{Z}}

\newcommand\N{\mathbb{N}}
\newcommand\kk{\Bbbk}

\newcommand\ba{\mathbf{a}}
\newcommand\bb{\mathbf{b}}
\newcommand\bc{\mathbf{c}}

\newcommand\op{\mathrm{op}}             % Opposite algebra
\newcommand\tr{\mathrm{tr}}

\newcommand\aff{\textup{aff}}
\newcommand\md{\textup{-mod}}

\DeclareMathOperator{\End}{End}

\DeclareMathOperator{\Hom}{Hom}

\DeclareMathOperator{\id}{id}
\DeclareMathOperator{\Ind}{Ind}
\DeclareMathOperator{\Res}{Res}
\DeclareMathOperator{\Span}{Span}

\newtheorem{theo}{Theorem}[section]
\newtheorem{prop}[theo]{Proposition}
\newtheorem{lem}[theo]{Lemma}
\newtheorem*{lem*}{Lemma}
\newtheorem{cor}[theo]{Corollary}

\theoremstyle{definition}
\newtheorem{defin}[theo]{Definition}
\newtheorem{rem}[theo]{Remark}
\newtheorem{eg}[theo]{Example}

\numberwithin{equation}{section}
\allowdisplaybreaks

\setenumerate[1]{label=(\alph*)}  % To change enumerate, reference with \ref{}
\setenumerate[2]{label=(\roman*)}

\newtoggle{comments}
\newtoggle{details}
\newtoggle{detailsnote}

\toggletrue{comments}   % To include comments (default is false)
\toggletrue{detailsnote}   % For including note about details toggle (default is false)

\iftoggle{comments}{%
  \newcommand{\acomments}[1]{
    \ \\
    {\color{red}
      \textbf{AS:} #1
    }
    \\
    }
  \newcommand{\dcomments}[1]{
    \ \\
    {\color{red}
      \textbf{DR:} #1
    }
    \\
    }
}{%
  \newcommand{\acomments}[1]{}
  \newcommand{\dcomments}[1]{}
}

\iftoggle{details}{%
  \newcommand{\details}[1]{
      \ \\
      {\color{OliveGreen}
        \textbf{Details:} #1
      }
      \\
  }
}{%
  \newcommand{\details}[1]{}
}

\begin{document}
%
%%%%%%%%%%%%%%%%%%%%%%%%%%%%%%%%%%%

\title{Quantum affine wreath algebras}

\author{Daniele Rosso}
\address[D.R]{
  Department of Mathematics and Actuarial Science \\
  Indiana University Northwest
}
\urladdr{\href{http://pages.iu.edu/~drosso/}{pages.iu.edu/~drosso/}}
\email{drosso@iu.edu}

\author{Alistair Savage}
\address[A.S.]{
  Department of Mathematics and Statistics \\
  University of Ottawa
}
\urladdr{\href{https://alistairsavage.ca}{alistairsavage.ca}, \textrm{\textit{ORCiD}:} \href{https://orcid.org/0000-0002-2859-0239}{orcid.org/0000-0002-2859-0239}}
\email{alistair.savage@uottawa.ca}
%\thanks{The author was supported by a Discovery Grant from the Natural Sciences and Engineering Research Council of Canada.}

\begin{abstract}
  To each symmetric algebra we associate a family of algebras that we call \emph{quantum affine wreath algebras}.  These can be viewed both as symmetric algebra deformations of affine Hecke algebras of type $A$ and as quantum deformations of affine wreath algebras.  We study the structure theory of these new algebras and their natural cyclotomic quotients.
\end{abstract}


\maketitle
\thispagestyle{empty}

%\tableofcontents

%=====================
\section{Introduction}
%=====================

Affine Hecke algebras and their degenerate versions are fundamental in the study of Lie algebras and quantum groups.  \emph{Affine wreath algebras}, whose systematic study was undertaken in \cite{Sav17}, provide a unifying and generalizing framework for various modified versions of degenerate affine Hecke algebras (of type $A$) appearing in the literature.\footnote{The terminology \emph{affine wreath product algebras} was used in \cite{Sav17}.  We drop the word ``product'' in the current paper for simplicity.}  (Certain cases of these algebras were also considered in \cite{KM15}.)  Affine wreath algebras also occur naturally as endomorphism algebras in the Frobenius Heisenberg categories of \cite{Sav18,RS17}.  It is natural to ask if such a general approach exists in the quantum (i.e.\ non-degenerate) setting.  The purpose of the current paper is to answer this question in the affirmative.

Fix a commutative ground ring $\kk$ and $z \in \kk$.  To any symmetric superalgebra $A$, we associate a \emph{quantum wreath algebra} $H_n(A,z)$.  The superalgebra $H_n(A,z)$ can be viewed as a $z$-deformation of the wreath algebra $A^{\otimes n} \rtimes S_n$, in the sense that $H_n(A,0) = A^{\otimes n} \rtimes S_n$.  Simultaneously, $H_n(A,z)$ can be thought of as an $A$-deformation of the Iwahori--Hecke algebra of type $A$.  In particular, taking $A = \kk = \C[q,q^{-1}]$ and $z=q-q^{-1}$ recovers the Iwahori--Hecke algebra.  We then define an affine version, the \emph{quantum affine wreath algebra} $H_n^\aff(A,z)$.  Again, this superalgebra can be viewed simultaneously as a $z$-deformation of the affine wreath algebras of \cite{Sav17} and as an $A$-deformation of the affine Hecke algebra of type $A$.

The quantum affine wreath algebras defined in the current paper unify and generalize existing analogs of affine Hecke algebras.  In particular, we have the following:
\begin{enumerate}
  \item When $A = \C$, the affine wreath algebra is the degenerate affine Hecke algebra of type $A$.  As noted above, when $A = \C[q,q^{-1}]$ and $z = q-q^{-1}$, the quantum affine wreath algebra is the affine Hecke algebra of type $A$.

  \item When $A$ is the group algebra of a finite group $G$, the affine wreath algebra is the wreath Hecke algebra of Wan and Wang \cite{WW08}.  When $G$ is a finite cyclic group, the quantum (affine) wreath algebra is the (affine) Yokonuma--Hecke algebra.  (See \cref{YHalg,affYHalg} for details.)  For more general groups, the quantum affine wreath algebra seems to be new.

  \item When $A$ is a certain skew-zigzag algebra (see \cite[\S3]{HK01} and \cite[\S5]{Cou16}), the corresponding affine wreath algebras appear in the endomorphism algebras of the categories constructed in \cite{CL12} to study Heisenberg categorification and the geometry of Hilbert schemes.  They were then also considered in \cite{KM15}, where they were related to imaginary strata for quiver Hecke algebras (also known as KLR algebras).
    \details{
      The skew-zigzag algebras considered in \cite{CL12} are symmetric (i.e.\ $\psi = \id$) in the super sense, but not in the usual (i.e.\ non-super) sense.  The zigzag algebras appearing in \cite{KM15} (which assumes all algebras are even) differ by some signs from those of \cite{CL12}, and are symmetric in the usual sense.
    }
    For this choice of $A$, the quantum affine wreath algebras of the current paper yield natural $z$-deformations of these affine zigzag algebras.  These deformations seem to be new.
\end{enumerate}

Despite their high level of generality, one can deduce a great deal of the structure of quantum affine wreath algebras.  Specializing the symmetric superalgebra $A$ then recovers known results in some cases and new results in others.  In addition, just as the affine wreath algebras appear as endomorphism algebras in the Frobenius Heisenberg categories of \cite{Sav18,RS17}, the quantum affine wreath algebras defined in the current paper appear as endomorphism algebras in the quantum Frobenius Heisenberg categories of \cite{BS19}.  In fact, this is one of the main motivations of the current paper.  The quantum Frobenius Heisenberg category acts on categories of modules for the quantum cyclotomic wreath algebras introduced here.  This action generalizes the action of the quantum Heisenberg category of \cite{BSW18} (see also \cite{LS13}) on categories of modules for cyclotomic Hecke algebras.

We now give an overview of the main results of the current paper.  We define the quantum (affine) wreath algebras in Section~\ref{sec:def} and discuss some natural symmetries.  In Section~\ref{sec:structure} we examine the structure theory of these algebras.  We first introduce natural Demazure operators which are useful in computations.  We then describe an explicit basis of $H_n^\aff(A,z)$ in \cref{affbasis}, and the center of $H_n^\aff(A,z)$ in \cref{center}.  Finally, we define natural Jucys--Murphy elements in \cref{subsec:JM} and give a Mackey Theorem for $H_n^\aff(A,z)$ in \cref{affMackey}.  In \cref{sec:cyclotomic} we turn our attention to cyclotomic quotients.  We define the quantum cyclotomic wreath algebra $H_n^f(A,z)$ associated to a monic polynomial $f$ with coefficients in the even part of the center $Z(A)$ of $A$.  These quotients are analogues of cyclotomic Hecke algebras.  We prove a basis theorem (\cref{theo:cyclobasis}) for these quotients and a cyclotomic Mackey Theorem (\cref{cycloMackey}).  Finally, we prove that the quantum cyclotomic wreath algebras are symmetric algebras and that $H_{n+1}^f(A,z)$ is a Frobenius extension of $H_n^f(A,z)$.

We expect that most of the results of the current paper can be generalized to the setting where $A$ is a Frobenius superalgebra, instead of a symmetric superalgebra (see \cref{Nakayama}).  This more general setting was treated in the degenerate case in \cite{Sav17} since the choice of $A$ to be the Clifford superalgebra, which is not symmetric in the super sense, yielded the affine Sergeev algebra (also called the degenerate affine Hecke--Clifford superalgebra).  However, in the quantum setting of the current paper we choose to focus on the case where $A$ is symmetric for simplicity.  In fact, the Clifford case is more naturally treated by considering an odd affinization of the quantum wreath algebra.  This will be explored in future work.

\iftoggle{detailsnote}{
\subsection*{Hidden details} For the interested reader, the tex file of the arXiv version of this paper includes hidden details of some straightforward computations and arguments that are omitted in the pdf file.  These details can be displayed by switching the \texttt{details} toggle to true in the tex file and recompiling.
}{}

%-----------------------------
\subsection*{Acknowledgements}
%-----------------------------

The first author was supported in this research by a Grant-in-Aid of Research and a Summer Faculty Fellowship from Indiana University Northwest. This research of the second author was supported by Discovery Grant RGPIN-2017-03854 from the Natural Sciences and Engineering Research Council of Canada.  We thank J.\ Brundan for helpful conversations.

%===================================
\section{Definitions\label{sec:def}}
%===================================

In this section we introduce our main objects of study.  Throughout the document, we fix a commutative ground ring $\kk$ of characteristic not equal to two.  (This assumption on the characteristic is not needed if one works in the non-super setting.)  We also fix an element $z \in \kk$.  All tensor products and algebras are over $\kk$ unless otherwise specified.  In addition, all algebras and modules are associative superalgebras and supermodules.  We drop the prefix ``super'' for simplicity.  For a homogeneous element $a$, we use the notation $\bar{a}$ to denote its parity.  We use $\N$ to denote the set of nonnegative integers.

%-----------------------------------
\subsection{Quantum wreath algebras}
%-----------------------------------

Fix a symmetric algebra $A$ with parity-preserving linear supersymmetric trace map $\tr \colon A \to \kk$.  (Here we consider $\kk$ to live in parity zero.)  In other words, the map
\[
  A \to \Hom_\kk(A,\kk),\quad a \mapsto \big( b \mapsto \tr(ab) \big),
\]
is a parity-preserving isomorphism of $\kk$-modules and, for homogeneous elements $a, b$,
\[
  \tr(ab) = (-1)^{\bar a \bar b} \tr(ba),\quad a,b \in A.
\]
We will assume that $A$ is free as a $\kk$-module.  So we have a basis $B$ with dual basis $\{b^\vee : b \in B\}$ defined by
\[
  \tr(a^\vee b) = \delta_{a,b},\quad a,b \in B.
\]
It follows from the supersymmetry of the trace that
\begin{equation} \label{ddual}
  (b^\vee)^\vee = (-1)^{\bar b} b,\quad b \in B.
\end{equation}

Fix $n \in \Z_{>0}$.  For $a \in A$ and $1 \le i \le n$, we define
\[
  a_i = 1^{\otimes (i-1)} \otimes a \otimes 1^{\otimes (n-i)} \in A^{\otimes n}.
\]

\begin{defin}[Quantum wreath algebra]
  For $n \in \N$, $n \ge 2$, we define the \emph{quantum wreath algebra} (or \emph{Frobenius Hecke algebra}) $H_n(A,z)$ to be the free product
  \[
    A^{\otimes n} \star \langle T_i : 1 \le i \le n-1 \rangle,
  \]
  (here the angled brackets mean the free associative algebra on the given generators) modulo the relations
  \begin{align}
    T_i T_j &= T_j T_i, &1 \le i,j \le n-1,\ |i-j| > 1, \label{farcomm}\\
    T_i T_{i+1} T_i &= T_{i+1} T_i T_{i+1}, &1 \le i \le n-2, \label{braid}\\
    T_i^2 &= z t_{i,i+1} T_i + 1, &1 \le i \le n-1, \label{quadratic} \\
    T_i \ba &= s_i(\ba) T_i, &\ba \in A^{\otimes n},\ 1 \le i \le n-1, \label{TF}
  \end{align}
  where
  \[
    t_{i,j} := \sum_{b \in B} b_i b_j^\vee,\quad
    1 \le i,j \le n-1,
  \]
  and $s_i(\ba)$ denotes the action of the simple transposition $s_i$ on $\ba$ by superpermutation of the factors.  It is straightforward to verify that $t_{i,j}$ does not depend on the choice of basis $B$.  We adopt the conventions that $H_1(A,z) := A$ and $H_0(A,z) := \kk$.
\end{defin}

For $w \in S_n$, we define
\[
  T_w = T_{i_1} T_{i_2} \dotsm T_{i_k},
\]
where $w = s_{i_1} s_{i_2} \dotsm s_{i_k}$ is a reduced decomposition.  Since the generators $T_i$ satisfy the braid relations \cref{farcomm,braid}, this definition is independent of the choice of reduced decomposition.

\begin{rem}
  In the degenerate case, it was shown in \cite[Lem.~3.2]{Sav17} that affine wreath algebras depend, up to isomorphism, only on the underlying algebra $A$, and not on the trace map.  However, in the quantum setting of the current paper, there do not seem to be obvious isomorphisms between quantum affine wreath algebras corresponding to the same algebra, but with different trace maps.
\end{rem}

\begin{eg}[Iwahori--Hecke algebras]
  If $A = \kk = \C[q,q^{-1}]$ and $z = q-q^{-1}$, then $H_n(A,z)$ is the Iwahori--Hecke algebra of type $A_{n-1}$.
\end{eg}

\begin{eg}[Yokonuma--Hecke algebras] \label{YHalg}
  Let $C_d$ be a cyclic group of order $d$.  If $\kk = \C[q,q^{-1}]$, $z=(q-q^{-1})/d$, and $A = \kk C_d$, with trace map given by projection onto the identity element of the group, then $H_n(\kk C_d,z)$ is the Yokonuma--Hecke algebra (see \cite[\S2.1]{CPd14}).
  \details{
    Let $y$ be a generator of $C_d$.  The isomorphism with the presentation in \cite[\S2.1]{CPd14} is given by:
    \[
      H_n(\kk C_d,z) \to Y_{d,n}(q),\quad
      T_i \mapsto g_i,\quad
      y_j \mapsto t_j,\quad
      1 \le i \le n-1,\ 1 \le j \le n.
    \]
  }
\end{eg}

It follows from \cref{ddual} that
\[
  t_{i,j} = t_{j,i},\quad 1 \le i,j \le n.
\]
Then, by \cref{TF}, we have
\[
  T_i t_{j,k} = t_{s_i(j), s_i(k)} T_i,\quad 1 \le i \le n-1,\ 1 \le j,k \le n.
\]
In particular,
\[
  T_i t_{i,i+1} = t_{i,i+1} T_i,\quad 1 \le i \le n-1.
\]
It then follows from \cref{quadratic} that the $T_i$ are invertible and we have a \emph{Frobenius skein relation}:
\begin{equation} \label{skein}
  T_i - T_i^{-1} = z t_{i,i+1},\quad 1 \le i \le n-1.
\end{equation}
We also have
\begin{equation} \label{teleport}
  \ba t_{j,k} = t_{j,k} s_{j,k}(\ba),\quad
  \ba \in A^{\otimes n},\ 1 \le i,j,k \le n,
\end{equation}
where $s_{j,k}$ is the transposition of $j$ and $k$.  For this reason, we call the $t_{i,j}$ \emph{teleporters}.  (In the string diagram formalism for monoidal categories, \cref{teleport} corresponds to tokens teleporting between strands.  See \cite[\S2.1]{Sav18}.)

%------------------------------------------
\subsection{Quantum affine wreath algebras}
%------------------------------------------

\begin{defin}[Quantum affine wreath algebra]
  For $n \in \N$, $n \ge 1$, we define the \emph{quantum affine wreath algebra} (or \emph{affine Frobenius Hecke algebra}) $H_n^\aff(A,z)$ to be the free product of algebras
  \[
    \kk[X_1^{\pm 1},\dotsc,X_n^{\pm 1}] \star H_n(A,z),
  \]
  modulo the relations
  \begin{align}
    T_i X_j &= X_j T_i, & 1 \le i \le n-1,\ 1 \le j \le n,\ j \ne i,i+1, \label{commute} \\
    T_i X_i T_i &= X_{i+1}, &1 \le i \le n-1, \label{inverse} \\
    X_i \ba &= \ba X_i, &1 \le i \le n,\ \ba \in A^{\otimes n}. \label{pna}
  \end{align}
  We define $H_{n,+}^\aff(A,z)$ to be the subalgebra of $H_n^\aff(A,z)$ generated by $H_n(A,z)$ together with $\kk[X_1, \dotsc, X_n]$ (no inverses).  We adopt the convention that $H_0^\aff(A,z) = H_{0,+}^\aff(A,z) := \kk$.
\end{defin}

\begin{eg}[Affine Hecke algebras] \label{eg:affHecke}
  If $A = \kk = \C[q,q^{-1}]$ and $z = q-q^{-1}$, then $H_n^\aff(A,z)$ is the affine Hecke algebra of type $A_{n-1}$.
\end{eg}

\begin{eg}[Affine Yokonuma--Hecke algebras] \label{affYHalg}
  In the setting of \cref{YHalg}, $H_n^\aff(\kk C_d,z)$ is the affine Yokonuma--Hecke algebra (see \cite[\S3.1]{CPd14}).
  \details{
    Let $y$ be a generator of $C_d$.  The isomorphism with the presentation in \cite[\S3.1]{CPd14} is given by:
    \[
      H_n^\aff(\kk C_d,z) \to \hat{Y}_{d,n}(q),\quad
      T_i \mapsto g_i,\quad
      y_j \mapsto t_j,\quad
      X_j \mapsto X_j,\quad
      1 \le i \le n-1,\quad 1 \le j \le n.
    \]
  }
\end{eg}

\begin{rem} \label{Nakayama}
  One can work in the more general setting where $A$ is a Frobenius algebra.  In general, there exists a Nakayama automorphism $\psi \colon A \to A$ such that $\tr(ab) = (-1)^{\bar{a} \bar{b}} \tr(b\psi(a))$ for all $a,b \in A$.  Then we modify the relation \cref{pna} to be $\ba X_i = X_i \psi_i(\ba)$, where $\psi_i = 1^{\otimes (i-1)} \otimes \psi \otimes 1^{\otimes (n-i)}$.  In the current paper we focus on the symmetric case, where $\psi=1$, for simplicity.  However, it is this more general setting that motivates our use of the word ``Frobenius'' in some of our terminology.
\end{rem}

%----------------------
\subsection{Symmetries}
%----------------------

It is straightforward to verify that we have an algebra automorphism of $H_n^\aff(A,z)$ given by
\begin{equation}
  X_i^{\pm 1} \mapsto X_{n+1-i}^{\pm 1},\quad
  a_i \mapsto a_{n+1-i},\quad
  T_j \mapsto -T_{n-j}^{-1} = -T_{n-j}+zt_{n-j,n-j+1},
\end{equation}
for all $1\leq i\leq n$, $1\leq j\leq n-1$, $a\in A$.

Any algebra automorphism $\xi \colon A \to A$ preserving the trace (i.e.\ $\tr \circ \xi = \tr$) induces an algebra automorphism of $H_n^\aff(A,z)$ given by
\begin{equation}
  X_i^{\pm 1} \mapsto X_{i}^{\pm 1},\quad
  \ba \mapsto \xi^{\otimes n}(\ba),\quad
  T_j\mapsto T_j,
\end{equation}
for all $1\leq i\leq n$, $1\leq j\leq n-1$, $\ba\in A^{\otimes n}$.
\begin{lem}
  Suppose that $\tau \colon A\to A^\op$ is an isomorphism of symmetric algebras (i.e.\ an algebra isomorphism preserving the trace map).  Then the map
  \[
    \widehat{\tau} \colon H_n^\aff(A,z) \to H_n^\aff(A,z)^\op,\quad
    X_i^{\pm 1} \mapsto X_{i}^{\pm 1},\quad
    \ba \mapsto \tau^{\otimes n}(\ba),\quad
    T_j\mapsto T_j,
  \]
  for all $1\leq i\leq n$, $1\leq j\leq n-1$, $\ba\in A^{\otimes n}$, is an isomorphism of algebras.
\end{lem}

\begin{proof}
  It is straightforward to verify that $\widehat{\tau}$ preserves the defining relations of $H_n^\aff(A,z)$, once it is noted that $\widehat{\tau}(t_{i,i+1})=t_{i,i+1}$ (see the proof of \cite[Lem.~3.9]{Sav17}).  So $\widehat{\tau}$ is indeed a homomorphism of algebras.  That $\widehat{\tau}$ is an isomorphism follows from the fact that it has inverse $\widehat{\tau^{-1}}$.
\end{proof}

Recall that the \emph{center} of $A$ is
\[
  Z(A) := \left\{ a \in A : ab = (-1)^{\bar a \bar b} ba \text{ for all } b \in A \right\}.
\]

\begin{lem}\label{rescale}
  Let $a \in Z(A)$ be invertible and even.  Then there exists a unique algebra automorphism $\zeta_a \colon H_n^\aff(A,z) \to H_n^\aff(A,z)$ given by
  \[
    \ba \mapsto \ba,\quad
    T_i \mapsto T_i, \quad
    X_j \mapsto a_j X_j,\quad
    X_j^{-1}\mapsto a^{-1}_j X^{-1}_j,
  \]
  for $\ba\in A^{\otimes n}$, $1 \le i \le n-1$, $1 \leq j \leq n$.
\end{lem}

\begin{proof}
  It is straightforward to verify that $\zeta_a$ preserves the defining relations of $H_n^\aff(A,z)$.
  \details{
    Since $\zeta_a$ is the identity on $H_n(A,z)$, relations \cref{farcomm,braid,quadratic,TF} are preserved.   We then just need to check \cref{commute,inverse,pna}.
    For $1 \leq i \leq n-1$ and $j \neq i,i+1$, we have
    \[
      \zeta_a(T_i X_j)
      = T_i a_j X_j
      \stackrel{\cref{TF}}{=} a_j T_i X_j
      \stackrel{\cref{commute}}{=} a_j X_j T_i
      = \zeta_a(X_jT_i),
    \]
    and so \cref{commute} is preserved.

    For $1 \leq i \leq n-1$, we have
    \[
      \zeta_a(T_iX_iT_i)
      = T_i a_i X_i T_i
      \stackrel{\cref{TF}}{=} a_{i+1} T_i X_i T_i
      \stackrel{\cref{inverse}}{=} a_{i+1} X_{i+1}
      = \zeta_a(X_{i+1}),
    \]
    which shows that \cref{inverse} is preserved.

    Finally, for $\ba\in A^{\otimes n}$ and $1 \leq j \leq n$, we have
    \[
      \zeta_a(X_j\ba)
      = a_j X_j \ba
      \stackrel{\cref{pna}}{=} a_j \ba X_j
      = \ba a_j X_j
      = \zeta_a(\ba X_j),
    \]
    where, in the third equality, we used the fact that $a$ is even and central.  Thus \cref{pna} is preserved.
  }
  Then, since $\zeta_a$ is invertible with inverse $\zeta_{a^{-1}}$, it is an automorphism.
\end{proof}

%==============================================
\section{Structure theory\label{sec:structure}}
%==============================================

In this section we examine the structure theory of quantum affine wreath algebras.  In particular, we describe a basis, the center, Jucys--Murphy elements, and a Mackey Theorem.

%------------------------------
\subsection{Demazure Operators}
%------------------------------

Let
\begin{equation}
  P_n = \kk[X_1^{\pm 1},\dotsc,X_n^{\pm 1}]
  \quad \text{and} \quad
  P_n(A) = A^{\otimes n} \otimes P_n \quad \text{(tensor product of algebras)}.
\end{equation}
We will use the notation $f,g$ to denote elements of $P_n(A)$ and the notation $p,q$ to denote elements of $P_n$.  By abuse of notation, for $f \in P_n(A)$ we will also denote by $f$ its image under the natural homomorphism $P_n(A)\to H_n^\aff(A,z)$.  In fact, it will follow from \cref{affbasis} that this homomorphism is injective, allowing us to view $P_n(A)$ as a subalgebra of $H_n^\aff(A,z)$.

We consider two different actions of $S_n$ on $P_n(A)$.  For $w \in S_n$ and $f \in P_n(A)$, we let $w(f)$ denote the action by permuting the $X_i$ and superpermuting the factors of $A^{\otimes n}$, i.e.\ the diagonal action on $A^{\otimes n} \otimes P_n$.  On the other hand, we let $^w f$ denote the action given by
\[
  ^w(\ba \otimes p) = \ba \otimes w(p).
\]
So this action permutes the $X_i$, but is $A^{\otimes n}$-linear.  Of course $w(p) = {}^w p$ for $p \in P_n$.

For $1 \le i \le n-1$, we have the \emph{Demazure operators}
\begin{equation}
  \Delta_i \colon P_n(A) \to P_n(A),\quad
  \Delta_i(f) = \frac{f - {^{s_i}f}}{1-X_iX_{i+1}^{-1}}.
\end{equation}
It is straightforward to verify that
\begin{equation} \label{leibniz}
  \Delta_i(fg) = \Delta_i(f)g + {}^{s_i}f \Delta_i(g),\quad f,g \in P_n(A).
\end{equation}
\details{
  We have
  \[
    \Delta_i(f g)
    = \frac{f g - {^{s_i} f} {^{s_i} g}}{1-X_iX_{i+1}^{-1}}
    = \frac{f g - {^{s_i} f} g + {^{s_i} f} g - {^{s_i} f} {^{s_i} g}}{1-X_iX_{i+1}^{-1}}
    = \Delta_i(f) g + {^{s_i}f} \Delta_i(g).
  \]
}
In particular,
\begin{equation} \label{sneak}
  \Delta_i(fg) = f \Delta_i(g) \quad
  \text{if} \quad {}^{s_i}f = f.
\end{equation}

\begin{lem}
  For $f \in P_n(A)$ and $1\leq i,j \leq n-1$, $|i-j|>1$, we have
  \begin{gather}
    ^{s_i} \Delta_i(f) = X_{i}X_{i+1}^{-1}\Delta_i(f), \quad
    \Delta_i \big( {}^{s_i} f \big) = -\Delta_i(f), \quad
    \Delta_i \big( {}^{s_j} f \big) = {}^{s_j} \Delta_i(f), \label{Delts} \\
    s_i \big( \Delta_i(f) \big) = - X_i X_{i+1}^{-1} \Delta_i \big( s_i(f) \big), \quad
    s_i \big( \Delta_j(f) \big) = \Delta_j \big( s_i(f) \big), \label{Delter} \\
    \Delta_i^2 = \Delta_i, \label{Delt1} \\
    \Delta_i \Delta_j = \Delta_j \Delta_i, \label{Delt2} \\
    \Delta_i\Delta_{i+1}\Delta_i=\Delta_{i+1}\Delta_i\Delta_{i+1} \label{Deltbraid} \quad \text{for } 1 \le i \le n-2.
  \end{gather}
\end{lem}

\begin{proof}
  The relations \cref{Delts} follow from straightforward computations.
  \details{
    We compute
    \begin{gather*}
      ^{s_i}\Delta_i(f)
      = \prescript{s_i}{}{\left(\frac{f-{}^{s_i}(f)}{1-X_iX_{i+1}^{-1}}\right)}
      = \frac{^{s_i}f-f}{1-X_{i+1}X_i^{-1}}
      = X_i X_{i+1}^{-1} \Delta_i(f),
      \\
      \Delta_i ({}^{s_i}f)
      = \frac{^{s_i}f-f}{1-X_iX_{i+1}^{-1}}
      = -\Delta_i(f), \\
      \Delta_i \big( {}^{s_j} f \big)
      = \frac{{}^{s_j} f - {}^{s_is_j} f}{1-X_iX_{i+1}^{-1}}
      = \prescript{s_j}{}{\left( \frac{f - {}^{s_i} f}{1-X_iX_{i+1}^{-1}} \right)}
      = {}^{s_j} \Delta_i(f).
    \end{gather*}
  }
  To see the first equation in \cref{Delter}, for $\ba \in A^{\otimes n}$, $p \in P_n$, we compute
  \begin{multline*}
    s_i \big( \Delta_i(\ba p) \big)
    = s_i(\ba) {}^{s_i} \Delta_i(p)
    \stackrel{\cref{Delts}}{=} X_i X_{i+1}^{-1} s_i(\ba) \Delta(p)
    \stackrel{\cref{Delts}}{=} - X_i X_{i+1}^{-1} s_i(\ba) \Delta({}^{s_i}p)
    \\
    = - X_i X_{i+1}^{-1} \frac{s_i(\ba p) - s_i(\ba) p}{1 - X_i X_{i+1}^{-1}}
    = - X_i X_{i+1}^{-1} \Delta_i \big( s_i(\ba p) \big).
  \end{multline*}
  The second equation in \cref{Delter} is straightforward.  To see \cref{Delt1}, we compute
  \[
    \Delta_i^2(f)
    = \frac{\Delta_i(f) - {}^{s_i}\Delta_i(f)}{1-X_iX_{i+1}^{-1}}
    \stackrel{\cref{Delts}}{=} \frac{\Delta_i(f)-X_iX_{i+1}^{-1}\Delta_i(f)}{1-X_iX_{i+1}^{-1}}
    = \Delta_i(f).
  \]
  For \cref{Delt2}, we compute
  \[
    \Delta_i \Delta_j(f)
    = \frac{\Delta_j(f) - {}^{s_i} \Delta_j(f)}{1-X_iX_{i+1}^{-1}}
    \stackrel{\cref{Delts}}{=} \frac{\Delta_j(f) - \Delta_j({}^{s_i}f)}{1-X_iX_{i+1}^{-1}}
    \stackrel{\cref{sneak}}{=} \Delta_j \Delta_i(f).
  \]

  Finally, to see \cref{Deltbraid}, we compute
  \begin{align*}
    \Delta_{i} &\Delta_{i+1} \Delta_{i} (f)
    = \Delta_i \Delta_{i+1} \left(\frac{f-{}^{s_i}f}{X_{i+1}-X_i}X_{i+1}\right) \\
    &= \Delta_i \left( \frac{f-{}^{s_i}f}{(X_{i+1}-X_i)(X_{i+2}-X_{i+1})} X_{i+1}X_{i+2} - \frac{{}^{s_{i+1}}f-{}^{s_{i+1}s_i}f}{(X_{i+2}-X_i)(X_{i+2}-X_{i+1})} X_{i+2}^2\right) \\
    &= \left( \frac{f-{}^{s_i}f}{(X_{i+1}-X_i)(X_{i+2}-X_{i+1})} X_{i+1}X_{i+2} - \frac{{}^{s_{i+1}}f-{}^{s_{i+1}s_i}f}{(X_{i+2}-X_i)(X_{i+2}-X_{i+1})} X_{i+2}^2 \right. \\
    &\quad - \left. \frac{f-{}^{s_i}f}{(X_{i+1}-X_i)(X_{i+2}-X_{i+1})} X_i X_{i+2} + \frac{{}^{s_is_{i+1}}f-{}^{s_is_{i+1}s_i}f}{(X_{i+2}-X_i)(X_{i+2}-X_{i+1})} X_{i+2}^2 \right) \frac{X_{i+1}}{X_{i+1}-X_i} \\
    &= \left( f - {}^{s_i}f - {}^{s_{i+1}}f + {}^{s_{i+1}s_i}f + {}^{s_is_{i+1}}f - {}^{s_is_{i+1}s_i}f \right) \frac{X_{i+1} X_{i+2}^2}{(X_{i+1}-X_i)(X_{i+2}-X_{i+1})(X_{i+2}-X_i)}.
  \end{align*}
  A similar computation for $\Delta_{i+1} \Delta_i \Delta_{i+1}(f)$ yields the same final expression.
  \details{
    We have
    \begin{align*}
      \Delta_{i+1} &\Delta_{i} \Delta_{i+1}(f)
      = \Delta_{i+1} \Delta_{i} \left( \frac{f-{}^{s_{i+1}}f}{X_{i+2}-X_{i+1}} X_{i+2} \right)\\
      &= \Delta_{i+1} \left( \frac{f-{}^{s_{i+1}}f}{(X_{i+1}-X_i)(X_{i+2}-X_{i+1})} X_{i+1} X_{i+2} - \frac{{}^{s_{i}}f-{}^{s_is_{i+1}}f}{(X_{i+2}-X_i)(X_{i+2}-X_{i+1}) }X_{i+1}X_{i+2} \right) \\
      &= \left( \frac{f-{}^{s_{i+1}}f}{(X_{i+1}-X_i)(X_{i+2}-X_{i+1})} X_{i+1} X_{i+2} - \frac{{}^{s_{i}}f-{}^{s_is_{i+1}}f}{(X_{i+2}-X_i)(X_{i+2}-X_{i+1})} X_{i+1} X_{i+2} + \right. \\
      &\quad - \left. \frac{f-{}^{s_{i+1}}f}{(X_{i+2}-X_i)(X_{i+2}-X_{i+1})} X_{i+1} X_{i+2} + \frac{{}^{s_{i+1}s_i}f-{}^{s_{i+1}s_is_{i+1}}f}{(X_{i+2}-X_i)(X_{i+2}-X_{i+1})} X_{i+1} X_{i+2} \right) \frac{X_{i+2}}{X_{i+2}-X_{i+1}} \\
      &= \left( f - {}^{s_i}f - {}^{s_{i+1}}f + {}^{s_{i+1}s_i}f + {}^{s_is_{i+1}}f - {}^{s_{i+1}s_is_{i+1}}f \right) \frac{X_{i+1} X_{i+2}^2}{(X_{i+1}-X_i)(X_{i+2}-X_{i+1})(X_{i+2}-X_i)}.
    \end{align*}
  }
\end{proof}

\begin{rem}
  The relations \cref{Delt1,Delt2,Deltbraid} imply that the $\Delta_i$ define an action of the $0$-Hecke algebra on $P_n(A)$.  Demazure operators first appeared in \cite{Dem74}.  Over the ring of integers, \cref{Delt1,Delt2,Deltbraid} are proved in \cite[Th.~2(a)]{Dem74} and \cite[(18)]{Dem74}.
\end{rem}

\begin{lem}\label{lem:passthrough}
  For all $f \in P_n(A)$ and $ 1 \le i \le n-1$, we have
  \begin{equation}\label{demazure}
    T_i f = s_i(f) T_i + z t_{i,i+1} \Delta_i(f).
  \end{equation}
\end{lem}

\begin{proof}
  It is straightforward to verify by direct computation that \cref{demazure} holds for $f = X_j^{\pm 1}$, $1 \le j \le n$.  For example,
  \[
    T_i X_i
    \stackrel{\cref{inverse}}{=} X_{i+1} T_i^{-1}
    \stackrel{\cref{skein}}{=}  X_{i+1} (T_i-zt_{i,i+1})
    = X_{i+1}T_i -z t_{i,i+1} X_{i+1}
    = X_{i+1}T_i+zt_{i,i+1}\Delta_i(X_i).
  \]
  \details{
    For $j \ne i,i+1$, we have
    \[
      T_i X_j
      \stackrel{\cref{commute}}{=} X_j T_i
      = X_j T_i + z t_{i,i+1} \Delta_i(X_j)
      \quad \text{and} \quad
      T_i X_j^{-1}
      \stackrel{\cref{commute}}{=} X_j^{-1} T_i
      = X_j^{-1} T_i + z t_{i,i+1} \Delta_i (X_j^{-1}).
    \]
    We also have
    \begin{gather*}
      T_i X_{i+1}
      \stackrel{\cref{skein}}{=}(T_i^{-1}+zt_{i,i+1})X_{i+1}
      \stackrel{\cref{inverse}}{=} X_{i}T_i+zt_{i,i+1}X_{i+1}
      = X_i T_i + z t_{i,i+1} \Delta_i(X_{i+1}),
      \\
      T_i X_i^{-1}
      \stackrel{\cref{skein}}{=} (T_i^{-1}+zt_{i,i+1}) X_i^{-1}
      \stackrel{\cref{inverse}}{=} X_{i+1}^{-1}T_i+zt_{i,i+1}X_i^{-1}
      = X_{i+1}^{-1} T_i+zt_{i,i+1}\Delta_i(X_i^{-1}),
      \\
      T_iX_{i+1}^{-1}
      \stackrel{\cref{inverse}}{=} X_i^{-1}T_i^{-1}
      \stackrel{\cref{skein}}{=} X_i^{-1}(T_i-zt_{i,i+1})
      = X_i^{-1}T_i-zt_{i,i+1}X_i^{-1}
      = X_i^{-1}T_i+zt_{i,i+1}\Delta_i(X_{i+1}^{-1}).
    \end{gather*}
  }
  Then, supposing the result holds for $p,q \in P_n$, we have
  \begin{multline*}
    T_i(pq)
    = s_i(p) T_i q + zt_{i,i+1} \Delta_i(p) q \\
    \stackrel{\cref{teleport}}{=} s_i(pq) T_i + z t_{i,i+1} s_i(p) \Delta_i(q) + zt_{i,i+1} \Delta_i(p) q
    \stackrel{\cref{leibniz}}{=} s_i(pq) T_i + z t_{i,i+1} \Delta_i(pq).
  \end{multline*}
  Since both sides of \cref{demazure} are $\kk$-linear in $f$ and the $X_j^{\pm 1}$ generate $P_n$ as a $\kk$-algebra, the result holds for all $p \in P_n$.  Now, for $\ba \in A^{\otimes n}$ and $p \in P_n$, we have
  \[
    T_i \ba p
    = s_i(\ba) T_i p
    = s_i(\ba) \left( s_i(p) T_i + z t_{i,i+1} \Delta_i(p) \right)
    \stackrel{\cref{teleport}}{=} s_i(\ba p) T_i + z t_{i,i+1} \Delta_i ( \ba p ).
  \]
  This completes the proof.
\end{proof}

\begin{lem} \label{snowy}
  Suppose $f \in P_n(A)$ and $w \in S_n$.  In $H_n^\aff(A,z)$, we have
  \begin{equation} \label{twf}
    T_w f = w(f) T_w + \sum_{u<w} f_u T_u,\qquad
    f T_w = T_w w^{-1}(f) + \sum_{u<w} T_u f'_u,
  \end{equation}
  for some $f_u,f_u' \in P_n(A)$.  Here $<$ denotes the strong Bruhat order on $S_n$.
\end{lem}

\begin{proof}
  This follows from \cref{demazure} by induction on the length of $w$.
\end{proof}

\begin{lem}
  For $k,\ell \in \Z$, $\ell > 0$, we have
  \begin{equation} \label{deltamon}
    \Delta_i \left( X_i^k X_{i+1}^{k+\ell} \right)
    = \sum_{r=0}^{\ell-1} X_i^{k+r} X_{i+1}^{k+\ell-r}
    = - \Delta_i \left( X_i^{k+\ell} X_{i+1}^k \right)
  \end{equation}
\end{lem}

\begin{proof}
  We have
  \[
    \Delta_i \left( X_i^k X_{i+1}^{k+\ell} \right)
    \stackrel{\cref{sneak}}{=} X_i^k X_{i+1}^k \Delta_i \left( X_{i+1}^\ell \right)
    = X_i^k X_{i+1}^{k+1} \frac{X_{i+1}^\ell - X_i^\ell}{X_{i+1}-X_i}
    = \sum_{r=0}^{\ell-1} X_i^{k+r} X_{i+1}^{k+\ell-r}.
  \]
  The second equation in \cref{deltamon} then follows from the second equation in \cref{Delts}.
\end{proof}

\begin{cor}\label{polyinv}
  For $1 \le i \le n-1$, we have $\Delta_i (\kk[X_1,\dotsc,X_n]) \subseteq \kk[X_1,\dotsc,X_n]$.
\end{cor}

\begin{proof}
  This follows from \cref{deltamon} and the fact that $\Delta_i$ is linear in the $X_j$, $j \ne i,i+1$, by \cref{sneak}.
\end{proof}

%-------------------------
\subsection{Basis Theorem}
%-------------------------

Our next goal is to give explicit bases for the quantum affine wreath algebra $H_n^\aff(A,z)$.  We do this by constructing a natural faithful representation.

\begin{lem}
  For $1\leq i\leq n-1$ and $\ell>0$,
  \begin{equation}\label{TiXd}
    T_iX_i^\ell T_i=X_{i+1}^\ell-zt_{i,i+1}\sum_{k=1}^{\ell-1}X_i^k X_{i+1}^{\ell-k}T_i
  \end{equation}
\end{lem}

\begin{proof}
  For $1\leq i\leq n-1$ and $\ell>0$, we have
  \begin{multline*}
    T_i X_i^\ell T_i
    \stackrel{\cref{demazure}}{=} \left( X_{i+1}^\ell T_i+zt_{i,i+1} \Delta_i \left( X_i^\ell \right) \right) T_i
    \\
    \stackrel{\cref{deltamon}}{=}X_{i+1}^\ell T_i^2-zt_{i,i+1}\sum_{k=0}^{\ell-1}X_i^k X_{i+1}^{\ell-k} T_i
    \stackrel{\cref{quadratic}}{=} X_{i+1}^\ell-zt_{i,i+1}\sum_{k=1}^{\ell-1}X_i^k X_{i+1}^{\ell-k}T_i. \qedhere
  \end{multline*}
\end{proof}

\begin{prop}\label{prop:action}
  Let $\mathcal{H}$ be the free $\kk$-module with basis $\{ T_w : w\in S_n\}$, and $V=P_n(A)\otimes \mathcal{H}$ a tensor product of $\kk$-modules. Then $V$ is an $H_n^\aff(A,z)$-module, with the action given by
  \begin{align*}
    f \cdot (g \otimes T_w) &= fg \otimes T_w,
    \\
    T_i \cdot (f \otimes T_w) &=
    \begin{cases}
      s_i(f) \otimes T_{s_iw} + z t_{i,i+1} \Delta_i(f)\otimes T_w & \text{ if }\ell(s_i w)>\ell(w), \\
      s_i(f)\otimes T_{s_iw} + z t_{i,i+1} X_{i+1}^{-1}\Delta_i(X_{i+1}f) \otimes T_w & \text{ if } \ell(s_iw)<\ell (w),
    \end{cases}
  \end{align*}
  for $f,g \in P_n(A)$, $w\in S_n$.  Here $\ell$ is the length function on $S_n$.
\end{prop}

\begin{proof}
  We need to check that the action satisfies the defining relations of $H_n^\aff(A,z)$.  Throughout this proof, $\ba \in A^{\otimes n}$, $f,g \in P_n(A)$, $w \in S_n$, $1 \le i \le n-1$, and $1 \le j \le n$.  The relation \cref{pna} is clearly satisfied.

  \medskip

  \noindent \emph{Relation~\cref{TF}}: We have
  \begin{align*}
    T_i \cdot (\ba \cdot (f \otimes T_w))
    &= T_i \cdot (\ba f \otimes T_w)
    \\ &=
    \begin{cases}
      s_i(\ba f) \otimes T_{s_iw} + z t_{i,i+1} \Delta_i(\ba f) \otimes T_w & \text{ if } \ell(s_i w)>\ell(w), \\
      s_i(\ba f) \otimes T_{s_iw} + z t_{i,i+1} X_{i+1}^{-1} \Delta_i (X_{i+1} \ba f) \otimes T_w & \text{ if } \ell(s_iw)<\ell (w)
    \end{cases}
    \\
    &\stackrel{\mathclap{\cref{teleport}}}{=}\ s_i(\ba)\cdot (T_i \cdot (f \otimes T_w)).
  \end{align*}

  \medskip

  \noindent \emph{Relation~\cref{commute}}:  If $j \ne i,i+1$, then
  \begin{align*}
    T_i \cdot ( X_j \cdot (f \otimes T_w))
    &= T_i \cdot (X_j f \otimes T_w) \\
    &=
    \begin{cases}
      s_i(X_jf) \otimes T_{s_iw} + z t_{i,i+1} \Delta_i(X_jf) \otimes T_w
      & \text{ if } \ell(s_i w) > \ell(w), \\
      s_i(X_jf) \otimes T_{s_iw} + z t_{i,i+1} X_{i+1}^{-1} \Delta_i(X_{i+1}X_jf) \otimes T_w
      & \text{ if } \ell(s_iw) <\ell (w)
    \end{cases}
    \\
    &\stackrel{\mathclap{\cref{sneak}}}{=}\ X_j \cdot (T_i \cdot (f \otimes T_w)).
  \end{align*}

  \medskip

  \noindent \emph{Relation~\cref{inverse}}:  First suppose that $\ell(s_iw)>\ell(w)$.  Then
  \begin{align*}
    T_i \cdot (X_i\cdot( f \otimes T_w))
    &= T_i \cdot (X_if \otimes T_w) \\
    &= s_i(X_if) \otimes T_{s_iw} + z t_{i,i+1} \Delta_i(X_if) \otimes T_w \\
    &\stackrel{\mathclap{\cref{leibniz}}}{=}\ X_{i+1} s_i(f) \otimes T_{s_iw} + z t_{i,i+1} \big( \Delta_i(X_i) f + X_{i+1} \Delta_i(f) \big) \otimes T_w \\
    &= X_{i+1} s_i(f) \otimes T_{s_i w} + z t_{i,i+1} X_{i+1} (-f + \Delta_i(f)) \otimes T_w \\
    &= X_{i+1} \cdot \big( ( T_i-zt_{i,i+1})\cdot (f \otimes T_w) \big) \\
    &\stackrel{\mathclap{\cref{skein}}}{=}\ X_{i+1} \cdot (T_i^{-1} \cdot (f \otimes T_w)).
  \end{align*}
  On the other hand, if $\ell(s_iw)<\ell(w)$, then
  \begin{align*}
    T_i\cdot (X_i\cdot( f \otimes T_w))
    &= T_i \cdot (X_if \otimes T_w) \\
    &= s_i(X_if) \otimes T_{s_iw} + z t_{i,i+1} X_{i+1}^{-1} \Delta_i(X_{i+1}X_if) \otimes T_w \\
    &\stackrel{\mathclap{\cref{sneak}}}{=}\ X_{i+1}s_i(f) \otimes T_{s_iw} + z t_{i,i+1} X_i \Delta_i(f) \otimes T_w \\
    &\stackrel{\mathclap{\cref{leibniz}}}{=}\ X_{i+1}s_i(f) \otimes T_{s_iw} + z t_{i,i+1} \big( \Delta_i(X_{i+1} f) - X_{i+1} f \big) \otimes T_w \\
    &= X_{i+1} \cdot \big( (T_i - z t_{i,i+1}) \cdot (f \otimes T_w) \big) \\
    &\stackrel{\mathclap{\cref{skein}}}{=}\ X_{i+1} \cdot (T_i^{-1} \cdot (f \otimes T_w)).
  \end{align*}

  \medskip

  \noindent \emph{Relation~\cref{quadratic}}:  First suppose $\ell(s_iw)>\ell(w)$, so that $\ell(s_i(s_iw))=\ell(w)<\ell(s_iw)$.  Then, using the fact that $s_i(t_{i,i+1}) = t_{i,i+1}$, we have
  \begin{align*}
    T_i &\cdot (T_i \cdot (f \otimes T_w)) \\
    &= T_i \cdot \big( s_i(f) \otimes T_{s_i w} + z t_{i,i+1} \Delta_i(f) \otimes T_w \big) \\
    &= f \otimes T_w + z t_{i,i+1} X_{i+1}^{-1} \Delta_i \big( X_{i+1} s_i(f) \big) \otimes T_{s_iw} + z t_{i,i+1} s_i \big( \Delta_i(f) \big) \otimes T_{s_i w} + z^2 t_{i,i+1}^2 \Delta_i^2(f) \otimes T_w \\
    &\stackrel{\mathclap{\cref{leibniz}}}{=}\ f \otimes T_w + z t_{i,i+1} \Big( s_i(f) + X_i X_{i+1}^{-1} \Delta_i \big( s_i(f) \big) \Big) \otimes T_{s_iw} + z t_{i,i+1} s_i \big( \Delta_i(f) \big) \otimes T_{s_i w} + z^2 t_{i,i+1}^2 \Delta_i^2(f) \otimes T_w \\
    &\stackrel[\mathclap{\cref{Delt1}}]{\mathclap{\cref{Delter}}}{=}\ f \otimes T_w + z t_{i,i+1} s_i(f) \otimes T_{s_iw} + z^2 t_{i,i+1}^2 \Delta_i(f) \otimes T_w \\
    &= (1 + zt_{i,i+1}T_i) \cdot (f \otimes T_w).
  \end{align*}
  The case $\ell(s_iw) < \ell(w)$ is similar.
  \details{
    If $\ell(s_iw) < \ell(w)$, then $\ell(s_i(s_iw))=\ell(w)>\ell(s_iw)$.  Thus
    \begin{align*}
      T_i &\cdot (T_i \cdot (f \otimes T_w)) \\
      &= T_i \cdot \big( s_i(f) \otimes T_{s_i w} + z t_{i,i+1} X_{i+1}^{-1} \Delta_i(X_{i+1}f) \otimes T_w \big) \\
      &= f \otimes T_w + z t_{i,i+1} \Delta_i \big( s_i(f) \big) \otimes T_{s_i w} + z t_{i,i+1} s_i \big( X_{i+1}^{-1} \Delta_i (X_{i+1} f) \big) \otimes T_{s_iw} + z^2 t_{i,i+1}^2 X_{i+1}^{-1} \big( \Delta_i^2 (X_{i+1} f) \big) \otimes T_w \\
      &\stackrel[\mathclap{\cref{Delter}}]{\mathclap{\cref{leibniz}}}{=}\ f \otimes T_w - z t_{i,i+1} X_i^{-1} X_{i+1} s_i \big( \Delta_i (f) \big) \otimes T_{s_i w} + z t_{i,i+1} s_i \big( f + X_i X_{i+1}^{-1} \Delta_i (f) \big) \otimes T_{s_iw} \\
      &\qquad \qquad + z^2 t_{i,i+1}^2 X_{i+1}^{-1} \big( \Delta_i^2 (X_{i+1} f) \big) \otimes T_w \\
      &\stackrel{\mathclap{\cref{Delt1}}}{=}\ f \otimes T_w + z t_{i,i+1} s_i(f) \otimes T_{s_iw} + z^2 t_{i,i+1}^2 X_{i+1}^{-1} \big( \Delta_i (X_{i+1} f) \big) \otimes T_w \\
      &= (1 + zt_{i,i+1}T_i) \cdot (f \otimes T_w).
    \end{align*}
  }

  \medskip

  \noindent \emph{Relation~\cref{farcomm}}:  Let $|i-j|>1$, so that $s_is_j=s_js_i$.  In order to handle several cases simultaneously, we use $\varepsilon$ and $\tau$ to denote elements of $\{0,1\}$ here.  Using \cref{Delter,sneak}, we have
  \begin{align*}
    T_i &\cdot \big( T_j \cdot (f \otimes T_w) \big) \\
    &= T_i \cdot \big( s_j(f) \otimes T_{s_jw} + z t_{j,j+1} X_{j+1}^{-\varepsilon} \Delta_j (X_{j+1}^\varepsilon f) \otimes T_w \big) \\
    &= (s_is_j)(f) \otimes T_{s_is_jw} + z t_{i,i+1} X_{i+1}^{-\tau} \Delta_i \big( X_{i+1}^\tau s_j(f) \big) \otimes T_{s_jw} \\
    &\qquad + z t_{j,j+1} s_i \big( X_{j+1}^{-\varepsilon} \Delta_j (X_{j+1}^\varepsilon f) \big) \otimes T_{s_iw} + z^2 t_{i,i+1} t_{j,j+1} X_{i+1}^{-\tau} \Delta_i \big( X_{i+1}^\tau X_{j+1}^{-\varepsilon} \Delta_j (X_{j+1}^\varepsilon) \big) \otimes T_w \\
    &= (s_js_i)(f) \otimes T_{s_js_iw} + z t_{i,i+1} s_j \big( X_{i+1}^{-\tau} \Delta_i ( X_{i+1}^\tau f) \big) \otimes T_{s_jw} \\
    &\qquad + z t_{j,j+1} X_{j+1}^{-\varepsilon} \Delta_j \big( X_{j+1}^\varepsilon s_i(f) \big) \otimes T_{s_iw} + z^2 t_{i,i+1} t_{j,j+1} X_{j+1}^{-\varepsilon} \Delta_j \big( X_{j+1}^\varepsilon X_{i+1}^{-\tau} \Delta_i (X_{i+1}^\tau) \big) \otimes T_w \\
    &= T_j \cdot \big( T_i \cdot (f \otimes T_w) \big).
  \end{align*}

  \medskip

  \noindent \emph{Relation~\cref{braid}}:  Verifying \cref{braid} is the most involved, and it occupies the remainder of the proof.  We first show that for $1\leq i\leq n-2$ and $1\leq j\leq n$ we have
  \begin{equation}\label{braidop1}
    (T_iT_{i+1}T_i-T_{i+1}T_iT_{i+1})X_j=X_{s_{i,i+2}(j)}(T_iT_{i+1}T_i-T_{i+1}T_iT_{i+1})
  \end{equation}
  as operators on $V$.  Clearly \cref{braidop1} holds for $j \neq i,i+1,i+2$ because \cref{commute} holds for the operators $T_i$ and $X_j$ on $V$, which we have already checked.  Notice also that \cref{demazure} holds in $\End(V)$ since the proof of that relation depends only on \cref{inverse,commute,quadratic}, which we have already verified.

  For $j = i$, as operators in $\End(V)$ we have
  \begin{align*}T_iT_{i+1}T_i X_i&= T_iT_{i+1}X_{i+1}(T_i-zt_{i,i+1}) \\
    &= T_iX_{i+2}(T_{i+1}-zt_{i+1,i+2})(T_i-zt_{i,i+1})\\
    &= X_{i+2}T_i(T_{i+1}T_i-zt_{i+1,i+2}T_i-zT_{i+1}t_{i,i+1}+z^2t_{i+1,i+2}t_{i,i+1})\\
    &=X_{i+2} (T_iT_{i+1}T_i-zT_it_{i+1,i+2}T_i-zT_iT_{i+1}t_{i,i+1}+z^2T_it_{i+1,i+2}t_{i,i+1})\\
    &\stackrel{\mathclap{\cref{TF}}}{=}\ X_{i+2} (T_iT_{i+1}T_i-zt_{i,i+2}T_i^2-zt_{i+1,i+2}T_iT_{i+1}+z^2t_{i,i+2}t_{i,i+1}T_i)\\
    &\stackrel{\mathclap{\cref{quadratic}}}{=}\ X_{i+2} (T_iT_{i+1}T_i-zt_{i,i+2}-zt_{i+1,i+2}T_iT_{i+1})
  \end{align*}
  and
  \begin{align*}
    T_{i+1}T_iT_{i+1} X_i
    &= T_{i+1}T_{i}X_{i}T_{i+1} \\
    &= T_{i+1}X_{i+1}(T_i-zt_{i,i+1})T_{i+1}\\
    &= X_{i+2}(T_{i+1}-zt_{i+1,i+2})(T_iT_{i+1}-zt_{i,i+1}T_{i+1})\\
    &= X_{i+2}(T_{i+1}T_iT_{i+1} - zT_{i+1}t_{i,i+1}T_{i+1} - zt_{i+1,i+2}T_iT_{i+1} + z^2t_{i+1,i+2} t_{i,i+1 }T_{i+1})\\
    &\stackrel{\mathclap{\cref{TF}}}{=}\ X_{i+2}(T_{i+1}T_iT_{i+1} - zt_{i,i+2}T_{i+1}^2 - zt_{i+1,i+2} T_i T_{i+1} + z^2 t_{i+1,i+2} t_{i,i+1} T_{i+1})\\
    &\stackrel{\mathclap{\cref{quadratic}}}{=}\ X_{i+2}(T_{i+1}T_iT_{i+1} - z^2t_{i,i+2}t_{i+1,i+2}T_{i+1} - zt_{i,i+2} - zt_{i+1,i+2}T_iT_{i+1} + z^2t_{i+1,i+2} t_{i,i+1} T_{i+1})\\
    &\stackrel{\mathclap{\cref{teleport}}}{=}\ X_{i+2}(T_{i+1}T_iT_{i+1}-zt_{i,i+2}-zt_{i+1,i+2}T_iT_{i+1}).
  \end{align*}
  So \cref{braidop1} holds for $j=i$.  The cases $j=i+1,i+2$ are similar.
  \details{
    For $j=i+1$, we have
    \begin{align*}
      T_iT_{i+1}T_i X_{i+1} &= T_iT_{i+1}(X_iT_i+zt_{i,i+1}X_{i+1})\\
      &\stackrel{\mathclap{\cref{TF}}}{=}\ T_iX_iT_{i+1}T_i+zt_{i+1,i+2}T_iT_{i+1}X_{i+1} \\
      &= X_{i+1}(T_iT_{i+1}T_i-zt_{i,i+1}T_{i+1}T_i)+zt_{i+1,i+2}X_{i+2}T_i(T_{i+1}-zt_{i+1,i+2})
    \end{align*}
    and
    \begin{align*}
      T_{i+1}T_iT_{i+1}X_{i+1}&=T_{i+1}T_iX_{i+2}(T_{i+1}-zt_{i+1,i+2})\\
      &= T_{i+1}X_{i+2}T_i(T_{i+1}-zt_{i+1,i+2})\\
      &=(X_{i+1}T_{i+1}+zt_{i+1,i+2}X_{i+2})T_i(T_{i+1}-zt_{i+1,i+2})\\
      &= X_{i+1}T_{i+1}T_i(T_{i+1}-zt_{i+1,i+2})+zt_{i+1,i+2}X_{i+2}T_i(T_{i+1}-zt_{i+1,i+2})\\
      &\stackrel{\mathclap{\cref{TF}}}{=}\ X_{i+1}(T_{i+1}T_iT_{i+1} - zt_{i,i+1}T_{i+1}T_i) + zt_{i+1,i+2} X_{i+2} T_i (T_{i+1}-zt_{i+1,i+2}).
    \end{align*}
    So \cref{braidop1} holds for $j=i+1$.  Finally, for $j=i+2$ we have
    \begin{align*}T_iT_{i+1}T_iX_{i+2}
      &= T_iT_{i+1}X_{i+2}T_i \\
      &= T_iX_{i+1}T_{i+1}T_i+zT_it_{i+1,i+2}X_{i+2}T_i \\
      &\stackrel{\mathclap{\cref{TF}}}{=}\ (X_iT_i+zt_{i,i+1}X_{i+1})T_{i+1}T_i+zt_{i,i+2}X_{i+2}T_i^2 \\
      &\stackrel{\mathclap{\cref{quadratic}}}{=}\  X_iT_iT_{i+1}T_i+zt_{i,i+1}X_{i+1}T_{i+1}T_i+z^2t_{i,i+2}t_{i,i+1}X_{i+2}T_i+zt_{i,i+2}X_{i+2}
    \end{align*}
    and
    \begin{align*}
      T_{i+1}T_iT_{i+1}X_{i+2}
      &= T_{i+1}T_iX_{i+1}T_{i+1}+zT_{i+1}T_it_{i+1,i+2}X_{i+2}\\
      &\stackrel{\mathclap{\cref{TF}}}{=}\ T_{i+1} (X_iT_i+zt_{i,i+1}X_{i+1}) T_{i+1} + z t_{i,i+1} T_{i+1} X_{i+2} T_i \\
      &= T_{i+1}X_iT_iT_{i+1} + zT_{i+1}t_{i,i+1}X_{i+1}T_{i+1} + zt_{i,i+1}(X_{i+1}T_{i+1} + zt_{i+1,i+2}X_{i+2})T_i \\
      &\stackrel[\mathclap{\cref{inverse}}]{\mathclap{\cref{TF}}}{=}\ X_i T_{i+1} T_i T_{i+1} + zt_{i,i+2} X_{i+2} + zt_{i,i+1} X_{i+1} T_{i+1}T_i + z^2t_{i,i+1} t_{i+1,i+2} X_{i+2}T_i \\
      &\stackrel{\mathclap{\cref{teleport}}}{=}\ X_i T_{i+1} T_i T_{i+1} + zt_{i,i+2} X_{i+2} + zt_{i,i+1} X_{i+1} T_{i+1} T_i + z^2 t_{i,i+2} t_{i,i+1} X_{i+2} T_i.
    \end{align*}
  }
  Notice that \cref{braidop1} also implies
  \begin{equation}\label{braidop2}
    X_{s_{i,i+2}(j)}^{-1}(T_iT_{i+1}T_i-T_{i+1}T_iT_{i+1})=(T_iT_{i+1}T_i-T_{i+1}T_iT_{i+1})X_j^{-1}.
  \end{equation}
  Then \cref{braidop1,braidop2,TF} imply that, as operators on $V$,
  \[
    (T_i T_{i+1} T_i - T_{i+1} T_i T_{i+1}) f
    = s_{i,i+2}(f) (T_i T_{i+1} T_i - T_{i+1} T_i T_{i+1})
    \quad \text{for all } f \in P_n(A).
  \]
  Thus, for all $f \in P_n(A)$ and $w \in S_n$, we have
  \[
    (T_i T_{i+1} T_i - T_{i+1} T_i T_{i+1}) \cdot (f \otimes T_w)
    = s_{i,i+2}(f) \cdot \big( (T_i T_{i+1} T_i - T_{i+1} T_i T_{i+1}) \cdot (1 \otimes T_w) \big).
  \]
  Hence, to prove \cref{braid}, it suffices to prove that
  \begin{equation} \label{couch}
    (T_i T_{i+1} T_i) \cdot (1 \otimes T_w)
    = (T_{i+1} T_i T_{i+1}) \cdot (1 \otimes T_w)
    \quad \text{for } w \in S_n.
  \end{equation}
  For the remainder of the proof, to simplify the notation, we write $T_w$ for $1 \otimes T_w$ and we omit $\cdot$ from the notation for the action.  We also adopt the convention that operators are applied in order from right to left.  For example, $T_i T_{i+1} T_i T_w = T_i (T_{i+1} (T_i T_w))$.

  It follows immediately from the definition of the action that we have
  \begin{equation} \label{meatballs}
    T_j T_w =
    \begin{cases}
      T_{s_j w} & \text{if } \ell(s_j w) > \ell(w), \\
      T_{s_j w} + z t_{j,j+1} T_w & \text{if } \ell(s_j w) < \ell(w).
    \end{cases}
  \end{equation}
  We split the proof of \cref{couch} into the following cases:
  \begin{enumerate}
    \item \label{length1} $\ell(s_i s_{i+1} s_i w) = \ell(w) + 3$,
    \item \label{length2} $\ell(s_i s_{i+1} s_i w) = \ell(w) + 1$,
    \item \label{length3} $\ell(s_i s_{i+1} s_i w) = \ell(w) - 1$,
    \item \label{length4} $\ell(s_i s_{i+1} s_i w) = \ell(w) - 3$.
  \end{enumerate}
  In case~\cref{length1}, we have
  \[
    \ell(w) < \ell(s_i w) < \ell(s_{i+1} s_i w) < \ell(s_i s_{i+1} s_i w)
    \quad \text{and} \quad
    \ell(w) < \ell(s_{i+1} w) < \ell(s_i s_{i+1} w) < \ell(s_{i+1} s_i s_{i+1} w).
  \]
  Thus
  \[
    T_i T_{i+1} T_i T_w
    = T_{s_i s_{i+1} s_i w}
    = T_{s_{i+1} s_i s_{i+1} w}
    = T_{i+1} T_i T_{i+1} T_w.
  \]

  In case~\cref{length2}, we have, without loss of generality,
  \[
    \ell(w) > \ell(s_i w) < \ell(s_{i+1} s_i w) < \ell(s_i s_{i+1} s_i w)
    \quad \text{and} \quad
    \ell(w) < \ell(s_{i+1} w) < \ell(s_i s_{i+1} w) > \ell(s_{i+1} s_i s_{i+1} w).
  \]
  (The other possibility is obtained by interchanging $i$ and $i+1$.)  Then we have a reduced word $w = s_i v$ and so
  \[
    T_i T_{i+1} T_i T_w
    = T_i T_{i+1} T_i T_i T_v
    \stackrel{\cref{quadratic}}{=} z T_i T_{i+1} t_{i,i+1} T_i T_v + T_i T_{i+1} T_v
    \stackrel{\cref{teleport}}{=} z t_{i+1,i+2} T_{s_i s_{i+1} s_i v} + T_{s_i s_{i+1} v}
  \]
  and
  \begin{multline*}
    T_{i+1} T_i T_{i+1} T_w
    = T_{i+1} T_{s_i s_{i+1} s_i v}
    = T_{i+1} T_{s_{i+1} s_i s_{i+1} v}
    \\
    = T_{i+1} T_{i+1} T_{s_i s_{i+1} v}
    \stackrel{\cref{quadratic}}{=} z t_{i+1,i+2} T_{s_{i+1} s_i s_{i+1} v} + T_{s_i s_{i+1} v}.
  \end{multline*}
  So \cref{couch} holds.

  In case~\cref{length3}, we have, without loss of generality,
  \[
    \ell(w) > \ell(s_i w) > \ell(s_{i+1} s_i w) < \ell(s_i s_{i+1} s_i w)
    \quad \text{and} \quad
    \ell(w) < \ell(s_{i+1} w) < \ell(s_i s_{i+1} w) > \ell(s_{i+1} s_i s_{i+1} w).
  \]
  (The other possibility is obtained by interchanging $i$ and $i+1$.)  Then we have a reduced word $w = s_i s_{i+1} v$ and so
  \begin{align*}
    T_i T_{i+1} T_i T_w
    &= T_i T_{i+1} T_i T_i T_{i+1} T_{v} \\
    &\stackrel{\mathclap{\cref{quadratic}}}{=}\ z T_i T_{i+1} t_{i,i+1} T_i T_{i+1} T_{v} + T_i T_{i+1} T_{i+1} T_{v} \\
    &\stackrel[\mathclap{\cref{quadratic}}]{\mathclap{\cref{teleport}}}{=}\ z t_{i+1,i+2} T_i  T_{s_{i+1} s_i s_{i+1} v} + z T_i t_{i+1,i+2} T_{i+1} T_{v} + T_{s_iv} \\
    &\stackrel{\mathclap{\cref{teleport}}}{=}\ z t_{i+1,i+2} T_i T_i T_{i+1} T_i T_{v} + z t_{i,i+2} T_{s_i s_{i+1} v} + T_{s_iv} \\
    &\stackrel{\mathclap{\cref{quadratic}}}{=}\ z^2 t_{i+1,i+2} t_{i,i+1} T_{s_i s_{i+1} s_i v} + z t_{i+1,i+2} T_{s_{i+1} s_i v} + z t_{i,i+2} T_{s_i s_{i+1} v} + T_{s_iv}
  \end{align*}
  and
  \begin{align*}
    T_{i+1} T_i T_{i+1} T_w
    &= T_{i+1} T_i T_{s_{i+1} s_i s_{i+1} v} \\
    &= T_{i+1} T_i T_{s_i s_{i+1} s_i v} \\
    &= T_{i+1} T_i T_i T_{i+1} T_i T_{v} \\
    &\stackrel{\mathclap{\cref{quadratic}}}{=}\ z T_{i+1} t_{i,i+1} T_i T_{i+1} T_i T_{v} + T_{i+1} T_{i+1} T_i T_{v} \\
    &\stackrel[\mathclap{\cref{quadratic}}]{\mathclap{\cref{TF}}}{=}\ z t_{i,i+2} T_{i+1} T_{i+1} T_i T_{i+1} T_{v} + z t_{i+1,i+2} T_{s_{i+1} s_i v} + T_{s_i v} \\
    &\stackrel{\mathclap{\cref{quadratic}}}{=}\ z^2 t_{i,i+2} t_{i+1,i+2} T_{s_{i+1} s_i s_{i+1} v} + z t_{i,i+2} T_{s_i s_{i+1} v} + z t_{i+1,i+2} T_{s_{i+1} s_i v} + T_{s_i v} \\
    &\stackrel{\mathclap{\cref{teleport}}}{=}\ z^2 t_{i+1,i+2} t_{i,i+1} T_{s_i s_{i+1} s_i v} + z t_{i+1,i+2} T_{s_{i+1} s_i v} + z t_{i,i+2} T_{s_i s_{i+1} v} + T_{s_iv}.
  \end{align*}
  So \cref{couch} holds.

  The case~\cref{length4} is similar and so will be omitted.
  \details{
    Finally, in case~\cref{length4}, we have
    \[
      \ell(w) > \ell(s_iw) > \ell(s_{i+1}s_iw) > \ell(s_is_{i+1}s_iw)
      \quad \text{and} \quad
      \ell(w) > \ell(s_{i+1}w) > \ell(s_is_{i+1}w) > \ell(s_{i+1}s_is_{i+1}w).
    \]
    So we can write a reduced expression $w=s_is_{i+1}s_iv=s_{i+1}s_is_{i+1}v$.  Then we have
    \begin{align*}
      T_i &T_{i+1} T_i T_w \\
      &= T_i T_{i+1} T_i T_i T_{i+1} T_iT_{v} \\
      &\stackrel{\mathclap{\cref{quadratic}}}{=}\ z T_i T_{i+1} t_{i,i+1} T_i T_{i+1} T_i T_{v} + T_i T_{i+1} T_{i+1} T_i T_{v} \\
      &\stackrel[\mathclap{\cref{teleport}}]{\mathclap{\cref{quadratic}}}{=}\ z t_{i+1,i+2} T_i T_{i+1} T_i T_{i+1} T_i T_{v} + z T_i t_{i+1,i+2} T_{i+1} T_i T_{v} + T_i T_i T_{v} \\
      &\stackrel[\mathclap{\cref{teleport}}]{\mathclap{\cref{quadratic}}}{=}\ z t_{i+1,i+2} T_i T_{i+1} T_{i+1} T_i T_{i+1} T_{v} + z t_{i,i+2} T_w + z t_{i,i+1} T_i T_{v} + T_{v} \\
      &\stackrel{\mathclap{\cref{quadratic}}}{=}\ z^2 t_{i+1,i+2} T_i t_{i+1,i+2} T_{i+1} T_i T_{i+1} T_{v} + z t_{i+1,i+2} T_i T_i T_{i+1} T_{v} + z t_{i,i+2} T_w + z t_{i,i+1} T_{s_i v} + T_{v} \\
      &\stackrel[\mathclap{\cref{teleport}}]{\mathclap{\cref{quadratic}}}{=}\ z^2 t_{i+1,i+2} t_{i,i+2} T_i T_i T_{i+1} T_i T_{v} + z^2 t_{i+1,i+2} t_{i,i+1} T_i T_{i+1} T_{v} + z t_{i+1,i+2} T_{i+1} T_{v} + z t_{i,i+2} T_w \\
      &\qquad \qquad + z t_{i,i+1} T_{s_i v} + T_{v} \\
      &\stackrel{\mathclap{\cref{quadratic}}}{=}\ z^3 t_{i+1,i+2} t_{i,i+2} t_{i,i+1} T_w + z^2 t_{i+1,i+2} t_{i,i+2} T_{s_{i+1} s_i v} + z^2 t_{i+1,i+2} t_{i,i+1} T_{s_i s_{i+1} v} + z t_{i+1,i+2} T_{s_{i+1} v} \\
      &\qquad \qquad + z t_{i,i+2} T_w + z t_{i,i+1} T_{s_i v} + T_{v}
    \end{align*}
    and
    \begin{align*}
      T_{i+1} &T_i T_{i+1} T_w \\
      &= T_{i+1} T_i T_{i+1} T_{i+1} T_i T_{i+1} T_{v} \\
      &\stackrel{\mathclap{\cref{quadratic}}}{=}\ z T_{i+1} T_i t_{i+1,i+2} T_{i+1} T_i T_{i+1} T_{v} + T_{i+1} T_i T_i T_{i+1} T_{v} \\
      &\stackrel[\mathclap{\cref{teleport}}]{\mathclap{\cref{quadratic}}}{=}\ z t_{i,i+1} T_{i+1} T_i T_i T_{i+1} T_i T_{v} + z T_{i+1} t_{i,i+1} T_i T_{i+1} T_{v} + T_{i+1} T_{i+1} T_{v} \\
      &\stackrel[\mathclap{\cref{teleport}}]{\mathclap{\cref{quadratic}}}{=}\ z^2 t_{i,i+1} T_{i+1} t_{i,i+1} T_i T_{i+1} T_i T_{v} + z t_{i,i+1} T_{i+1} T_{i+1} T_i T_{v} + z t_{i,i+2} T_w + zt_{i+1,i+2} T_{i+1} T_{v} + T_{v} \\
      &\stackrel[\mathclap{\cref{teleport}}]{\mathclap{\cref{quadratic}}}{=}\ z^2 t_{i,i+1} t_{i,i+2} T_{i+1} T_{i+1} T_i T_{i+1} T_{v} + z^2 t_{i,i+1} t_{i+1,i+2} T_{s_{i+1} s_i v} + z t_{i,i+1} T_{s_i v} + z t_{i,i+2} T_w \\
      &\qquad \qquad + z t_{i+1,i+2} T_{s_{i+1}v} + T_{v} \\
      &\stackrel{\mathclap{\cref{quadratic}}}{=}\ z^3 t_{i,i+1} t_{i,i+2} t_{i+1,i+2} T_w + z^2 t_{i,i+1} t_{i,i+2} T_{s_i s_{i+1} v} + z^2 t_{i,i+1} t_{i+1,i+2} T_{s_{i+1} s_i v} + z t_{i,i+1} T_{s_i v} \\
      &\qquad \qquad + z t_{i,i+2} T_w + z t_{i+1,i+2} T_{s_{i+1}v} + T_{v}.
    \end{align*}
    It follows from \cref{teleport} that
    \[
      t_{i+1,i+2} t_{i,i+2} t_{i,i+1} = t_{i,i+1} t_{i,i+2} t_{i+1,i+2},\quad
      t_{i+1,i+2} t_{i,i+2} = t_{i,i+1} t_{i+1,i+2},\quad
      t_{i+1,i+2} t_{i,i+1} = t_{i,i+1} t_{i,i+2}.
    \]
    Hence \cref{couch} holds.
  }
\end{proof}

\begin{rem}\label{curraction}
  Notice that if $p \in \kk[X_1,\dotsc,X_n]$, then $X_{i+1}^{-1}\Delta_i(X_{i+1}p) \in \kk[X_1,\dotsc, X_n]$ by \cref{deltamon}.  Then, also using \cref{polyinv}, it is easy to see that if we take $V_+\subsetneq V$ to be the space obtained by replacing $P_n(A)$ with $\kk[X_1,\dotsc,X_n]\otimes A$ in \cref{prop:action}, then $V_+$ is invariant under the action of $H^\aff_{n,+}(A,z)$.
\end{rem}

\begin{theo}[Basis Theorem for $H_n^\aff(A,z)$] \label{affbasis}
  The map
  \[
    V = P_n(A) \otimes \mathcal{H} \to H_n^\aff(A,z),\qquad
    f \otimes T_w \mapsto f T_w,
  \]
  is an isomorphism of $H_n^\aff(A,z)$-modules.
\end{theo}

\begin{proof}
  Let $D$ be a basis of $P_n(A)$, and let
  \[
    \mathcal{B}_1 = \{ f \otimes T_w : f \in D,\ w \in S_n \} \subseteq V, \quad
    \mathcal{B}_2 = \{ f T_w : f \in D,\ w \in S_n\} \subseteq H_n^\aff(A,z).
  \]
  Then $\mathcal{B}_1$ is a basis of $V$.  It follows from \cref{lem:passthrough} that $\mathcal{B}_2$ spans $H_n^\aff(A,z)$.  Furthermore, we have that $(f T_w) \cdot (1 \otimes T_1) = f \otimes T_w$, and so the elements of $\mathcal{B}_2$ are linearly independent, hence a basis.  Since $V$ is a cyclic module generated by $1\otimes T_e$, there is an $H_n^\aff(A,z)$-module homomorphism $H_n^\aff(A,z) \to V$, determined by $1 \mapsto 1\otimes T_e$. This map sends $f T_w \in\mathcal{B}_2$ to $f \otimes T_w \in \mathcal{B}_2$, hence it is an isomorphism because it gives a bijection of $\kk$-bases.
\end{proof}

For $\lambda=(\lambda_1,\dotsc,\lambda_n) \in \Z^n$, we let $X^{\lambda}=X_1^{\lambda_1} \dotsm X_n^{\lambda_n}$.  Recall that $B$ is a $\kk$-basis for $A$.

\begin{cor}\label{basis}
  The sets
  \[
    \{ \ba X^\lambda T_w : \ba \in B^{\otimes n},\ \lambda \in \Z^n,\ w\in S_n\}
    \quad \text{and} \quad
    \{ T_w \ba X^\lambda : \ba \in B^{\otimes n},\ \lambda \in \Z^n,\ w\in S_n\}
  \]
  are $\kk$-bases for $H_n^\aff(A,z)$.
\end{cor}

\begin{proof}
  It follows immediately from \cref{affbasis} that the first set is a basis. The fact that the second set is also a basis follows from \cref{twf} by induction on the length of $w$.
\end{proof}

\begin{cor}\label{currbasis}
  The sets
  \[
    \{ \ba X^{\lambda} T_w : \ba \in B^{\otimes n},\ \lambda\in \N^n,\ w\in S_n\}
    \quad \text{and} \quad
    \{ T_w \ba X^{\lambda} : \ba \in B^{\otimes n},\ \lambda\in \N^n,\ w\in S_n\}
  \]
  are $\kk$-bases for $H_{n,+}^\aff(A,z)$.
\end{cor}

\begin{proof}
  This uses the same reasoning as \cref{basis}, due to \cref{curraction}.
\end{proof}

\begin{rem}
  For the case of the affine Hecke algebras (see \cref{eg:affHecke}), \cref{basis} recovers a result of Lusztig \cite[Prop.~3.7]{Lus89}.  For affine Yokonuma--Hecke algebras (see \cref{affYHalg}), it was proved in \cite[Th.~4.4]{CPd16}.
\end{rem}

%-------------------------------------
\subsection{Description of the center}
%-------------------------------------

We now compute the center of quantum affine wreath algebras.  By \cref{pna}, we have that $P_n(Z(A)) = Z(A)^{\otimes n}\otimes P_n$ is a subalgebra of $P_n(A)$.

\begin{lem}\label{lem:centPn}
  The centralizer of $P_n$ in $H_n^\aff(A,z)$ is equal to $P_n(A)$.
\end{lem}

\begin{proof}
  By \cref{pna}, it is clear that elements of $P_n(A)$ commute with elements of $P_n$.  Now let $\alpha = \sum_{w\in S_n} f_w T_w \in H_n^\aff(A,z)$, where $f_w\in P_n(A)$ for all $w\in S_n$. Let $v\in S_n$ be a maximal element in the strong Bruhat order such that $f_{v}\neq 0$. Suppose $v\neq 1$, and let $1\leq i \leq n$ such that $v(i)\neq i$. Then, by \cref{twf} we have
  \[
    X_i \alpha -\alpha X_i
    = (X_i - X_{v(i)})f_{v}T_{v} + \sum_{u \ngeq v} g_u T_u,
  \]
  for some $g_u\in P_n(A)$. Thus, by \cref{affbasis}, $\alpha$ does not centralize $X_i$; hence the result.
\end{proof}

\begin{lem}\label{lem:Zinpnza}
  The centralizer of $P_n(A)$ in $H_n^\aff(A,z)$ is equal to $P_n(Z(A))$.
\end{lem}

\begin{proof}
  The centralizer of $P_n(A)$ inside $H_n^\aff(A,z)$ is contained in the centralizer of $P_n$, which by \cref{lem:centPn} is equal to $P_n(A)$. Hence the centralizer of $P_n(A)$ is equal to the center $Z(P_n(A))$.  Using \cref{pna}, we have
  \[
    Z(P_n(A))
    = Z\left( A^{\otimes n}\otimes P_n \right)
    = Z(A^{\otimes n})\otimes Z(P_n)
    = Z(A)^{\otimes n}\otimes P_n=P_n(Z(A)),
  \]
  where we use the fact that $Z(A^{\otimes n}) = Z(A)^{\otimes n}$ since $A$ is free over $\kk$.
\end{proof}

For any subset $Y \subseteq P_n(A)$, we define
\[
  Y^{S_n} = \{ f \in Y : w(f) = f \text{ for all } w \in S_n\}.
\]

\begin{theo} \label{center}
  We have $Z \big( H_n^\aff(A,z) \big) = P_n(Z(A))^{S_n}$.
\end{theo}

\begin{proof}
  Suppose $f \in Z(H_n^\aff(A,z))\subseteq P_n(Z(A))$.  For $1 \leq i \leq n-1$, \cref{twf} implies that
  \[
    T_i f = s_i(f) T_i + g
    \quad \text{for some } g \in P_n(A).
  \]
  Then
  \[
    f
    = T_i f T_i^{-1}
    = \big( s_i(f)T_i + g \big) T_i^{-1}
    \stackrel{\cref{skein}}{=} s_i(f) + g (T_i-zt_{i,i+1})
    = s_i(f) + g T_i - z g t_{i,i+1}.
  \]
  By \cref{basis}, we have $g=0$; hence $f=s_i(f)$.  Since this is true for all $1 \leq i \leq n-1$, it follows that $f \in P_n(Z(A))^{S_n}$.

  Now suppose $f \in P_n(Z(A))^{S_n}$.  For each $1 \le i \le n-1$ and $\ba \in Z(A)^{\otimes n}$, we have
  \[
    t_{i,i+1} \ba
    = \ba t_{i,i+1}
    \stackrel{\cref{teleport}}{=} t_{i,i+1} s_i(\ba).
  \]
  It follows that $t_{i,i+1} {}^{s_i} f = t_{i,i+1} s_i(f)$, and so
  \[
    t_{i,i+1} \Delta_i(f)
    = t_{i,i+1} \frac{f - {}^{s_i}f}{1-X_iX_{i+1}^{-1}}
    = t_{i,i+1} \frac{f-s_i(f)}{1-X_iX_{i+1}^{-1}}
    = 0.
  \]
  Thus, by \cref{demazure}, we have $T_i f = f T_i$ for all $1 \le i \le n-1$.  Since $f$ clearly commutes with all elements of $P_n(A)$, we have $f \in Z \big( H_n^\aff(A,z) \big)$.
\end{proof}

\begin{rem}
  For affine Hecke algebras (see \cref{eg:affHecke}), \cref{center} recovers a well-known description of the center (see \cite[Prop.~3.11]{Lus89}).  For affine Yokonuma--Hecke algebras (see \cref{affYHalg}), it recovers \cite[Th.~2.7]{CW15}.
\end{rem}

\begin{prop}
  Suppose $A'$ is a maximal commutative subalgebra of $A$.  Then $P_n(A')$ is a maximal commutative subalgebra of $H_n^\aff(A,z)$.
\end{prop}

\begin{proof}
  Suppose $\alpha \in H_n^\aff(A,z)$ commutes with all elements of $P_n(A')$.  By \cref{lem:centPn}, we have $\alpha \in P_n(A)$.  Thus $\alpha = \sum_{\lambda \in \Z^n} \ba_\lambda X^\lambda$, for some $\ba_\lambda \in A^{\otimes n}$.  Then, for all $\bb \in (A')^{\otimes n}$, we have
  \[
    \bb \alpha = \alpha \bb
    \implies \sum_{\lambda \in \Z^n} \bb \ba_\lambda X^\lambda
    = \sum_{\lambda \in \Z^n} \ba_\lambda \bb X^\lambda.
  \]
  Thus, by \cref{basis}, $\bb \ba_\lambda = \ba_\lambda \bb$ for all $\lambda$.  Since $A'$ is a maximal commutative subalgebra of $A$, this implies that $\ba_\lambda \in (A')^{\otimes n}$ for all $\lambda$.  Hence $\alpha \in P_n(A')$.
\end{proof}

%--------------------------------------------------
\subsection{Jucys--Murphy elements\label{subsec:JM}}
%--------------------------------------------------

Define the \emph{Jucys--Murphy} elements in $H_n(A,z)$ by
\begin{equation}\label{eq:jucys}
  J_1 = 1,\qquad
  J_i = T_{i-1}J_{i-1}T_{i-1} = T_{i-1} \dotsm T_2T_1^2T_2 \dotsm T_{i-1}, \qquad 2\leq i\leq n.
\end{equation}
These elements generalize the well-known Jucys--Murphy elements in the Iwahori--Hecke algebra, as well as the Jucys--Murphy elements of the Yokonuma--Hecke algebra introduced in \cite[(2.14)]{CPd14}.

\begin{prop}\label{prop:jucysmap}
  There is a surjective algebra homomorphism $H_n^\aff(A,z)\to H_n(A,z)$ defined by
  \[
    X_i \mapsto J_i,\qquad
    \ba \mapsto \ba,\qquad
    T_w \mapsto T_w,\quad
    1 \leq i\leq n,\ \ba \in A^{\otimes n},\ w\in S_n.
  \]
\end{prop}

\begin{proof}
  We need to check that this maps preserves the relations \cref{commute,inverse,pna}, in addition to the facts that $X_iX_j=X_jX_i$ and that $X_i$ is invertible for $1\leq i,j\leq n$.  Clearly $J_i$ is invertible for all $1\leq i\leq n$ because $T_k$ is invertible for all $1\leq k\leq n$. Also, \cref{inverse} follows from the definition of $J_i$.  Relation \cref{pna} is satisfied because of \cref{TF}, while the fact that $J_iJ_j=J_jJ_i$ for all $1\leq i,j\leq n$, in addition to relation \cref{commute}, follows from repeated use of \cref{farcomm,braid}.
\end{proof}

%--------------------------
\subsection{Mackey Theorem}
%--------------------------

For a composition $\mu = (\mu_1, \dotsc, \mu_r)$ of $n$, let
\[
  S_\mu = S_{\mu_1} \times \dotsb \times S_{\mu_r} \subseteq S_n
\]
denote the corresponding Young subgroup.  We then define the parabolic subalgebra $H_\mu(A,z) \subseteq H_n(A,z)$ to be the subalgebra generated by $A^{\otimes n}$ and $\{T_w : w\in S_\mu\}$.  We also define $H_\mu^\aff(A,z) \subseteq H_n^\aff(A,z)$ to be the subalgebra generated by $H_\mu(A,z)$ and $P_n$. So we have an isomorphism of algebras
\[
  H_\mu^\aff(A,z) \cong H_{\mu_1}^\aff(A,z) \otimes \dotsb \otimes H_{\mu_r}^\aff(A,z),
\]
and a parity-preserving isomorphism of $\kk$-modules
\[
  H_\mu^\aff(A,z) \cong P_n \otimes H_\mu(A,z).
\]
Let $D_{\mu,\nu}$ denote the set of minimal length $(S_\mu,S_\nu)$-double coset representatives in $S_n$. By \cite[Lem.~1.6(ii)]{DJ86}, for $\pi\in D_{\mu,\nu}$, $S_\mu \cap \pi S_{\nu}\pi^{-1}$ and $\pi^{-1} S_\mu \pi \cap S_\nu$ are Young subgroups of $S_n$; hence we can define compositions $\mu\cap \pi\nu$ and $\pi^{-1}\mu\cap \nu$ by
\[
  S_\mu\cap \pi S_{\nu}\pi^{-1}
  = S_{\mu\cap\pi\nu}
  \quad \text{and} \quad
  \pi^{-1} S_\mu \pi\cap S_\nu
  = S_{\pi^{-1}\mu\cap \nu}.
\]
Furthermore, the map $w \mapsto \pi^{-1} w\pi$ restricts to a length preserving isomorphism
\[
  S_{\mu\cap \pi\nu}\to S_{\pi^{-1}\mu\cap \nu}
\]
which, due to the length-preserving property, induces an isomorphism of algebras
\[
  H_{\mu\cap \pi\nu}(A,z) \to H_{\pi^{-1}\mu\cap \nu}(A,z),\quad
  \ba \mapsto \pi^{-1}(\ba),\quad
  T_w \mapsto T_{\pi^{-1}w\pi},\quad
  \ba \in A^{\otimes n},\ w \in S_n.
\]

It is easy to verify that, for $\pi\in D_{\mu,\nu}$ and $s_i\in S_{\mu\cap\pi\nu}$, we have $\pi^{-1}(i+1) = \pi^{-1}(i)+1$, and hence $\pi^{-1}s_i\pi = s_{\pi^{-1}(i)}$.
\details{
  The fact that $s_i \in S_{\mu \cap \pi \nu} = S_\mu \cap \pi S_\nu \pi^{-1}$ implies that $i$ and $i+1$ are in the same subset of the partition $\mu$ (thought of as a partition of the set $\{1,2,\dotsc,n\}$ into the subsets $\{1,\dotsc,\mu_1\}$, $\{\mu_1+1,\dotsc, \mu_1+\mu_2\}$, etc.) and that $\pi^{-1}(i)$ and $\pi^{-1}(i+1)$ are in the same subset of the partition $\nu$.  On the other hand, $\pi \in D_{\mu,\nu}$ means that $\pi$ does not invert elements in the same subset of $\nu$ and $\pi^{-1}$ does not invert elements in the same subset of $\mu$.  It follows that $\pi^{-1}(i) < \pi^{-1}(i+1)$.  Suppose, towards a contradiction, that $\pi^{-1}(i) + 1 < \pi^{-1}(i+1)$.  Then there exists an integer $j$ with $\pi^{-1}(i) < j < \pi^{-1}(i+1)$.  Then  $\pi^{-1}(i)$, $j$, and $\pi^{-1}(i+1)$ are are in the same subset of the partition $\nu$.  Hence $\pi$ does not invert these elements, and so $i < \pi(j) < i+1$, which is clearly a contradiction.
}
Thus, for each $\pi\in D_{\mu,\nu}$, we have an algebra isomorphism
\begin{gather*}
  \varphi_{\pi^{-1}} \colon H_{\mu\cap \pi\nu}^\aff(A,z) \to H_{\pi^{-1}\mu\cap \nu}^\aff(A,z), \\
  \varphi_{\pi^{-1}}(T_w)=T_{\pi^{-1}w\pi},\quad
  \varphi_{\pi^{-1}}(f) = \pi^{-1}(f), \quad
  w \in S_{\mu\cap\pi\nu},\ f \in P_n(A).
\end{gather*}
\details{
  As noted above, $\varphi_{\pi^{-1}}$ induces an isomorphism $H_{\mu\cap \pi\nu}(A,z) \to H_{\pi^{-1}\mu\cap \nu}(A,z)$.  It is also clear that it induces an automorphism of $P_n$.  It obviously preserves the relation \cref{pna}.  Now suppose $j \ne i,i+1$.  Then $\pi^{-1} i \ne \pi^{-1} j$ and $\pi^{-1}(i)+1 = \pi^{-1}(i+1) \ne \pi^{-1} j$ by the above.  Thus
  \[
    \varphi_{\pi^{-1}} (T_i X_j)
    = T_{\pi^{-1} i} X_{\pi^{-1} j}
    \stackrel{\cref{commute}}{=} X_{\pi^{-1} j} T_{\pi^{-1} i}
    = \varphi_{\pi^{-1}} (X_j T_i),
  \]
  and so $\varphi_{\pi^{-1}}$ preserves \eqref{commute}.  Finally,
  \[
    \varphi_{\pi^{-1}} (T_i X_i T_i)
    = T_{\pi^{-1} i} X_{\pi^{-1} i} T_{\pi^{-1} i}
    \stackrel{\eqref{inverse}}{=} X_{\pi^{-1}(i) + 1}
    = X_{\pi^{-1}(i+1)}
    = \varphi_{\pi^{-1}}(X_{i+1}),
  \]
  and so $\varphi_{\pi^{-1}}$ preserves \eqref{inverse}.
}
If $N$ is a left $H_{\pi^{-1}\mu\cap \nu}^\aff(A,z)$-module, we denote by ${}^{\pi}N$ the left $H_{\mu\cap\pi\nu}^\aff(A,z)$-module with action given by
\[
  \alpha \cdot v
  = \varphi_{\pi^{-1}}(\alpha)v,\quad
  \alpha\in H_{\mu\cap\pi\nu}^\aff,\ v \in {}^\pi N=N.
\]
The inclusion $H_\mu^\aff(A,z) \subseteq H_n^\aff(A,z)$ gives induction and restriction functors
\[
  \Res^n_\mu \colon H_n^\aff(A,z)\md\to H_\mu^\aff(A,z)\md,\qquad
  \Ind^n_\mu \colon H_\mu^\aff(A,z)\md \to H_n^\aff(A,z)\md.
\]

\begin{theo}[Mackey Theorem for $H_n^\aff(A,z)$] \label{affMackey}
  Suppose that $M$ is an $H_n^\aff(A,z)$-module. Then $\Res^n_\mu \Ind^n_\nu M$ admits a filtration with subquotients evenly isomorphic to
  \[
    \Ind^\mu_{\mu\cap\pi\nu} \prescript{\pi}{}{\left(\Res^\nu_{\pi^{-1}\mu\cap\nu} M\right)},
  \]
  one for each $\pi\in D_{\mu,\nu}$.  Furthermore, the subquotients can be taken in any order refining the strong Bruhat order on $D_{\mu,\nu}$. In particular, $\Ind^\mu_{\mu\cap\nu} \Res^{\nu}_{\mu\cap\nu} M$ appears as a submodule.
\end{theo}

\begin{proof}
  The proof is essentially the same as the proofs of \cite[Thm 3.5.2]{Kle05} and \cite[Thm 14.5.2]{Kle05}; hence it will be omitted.
\end{proof}

%===================================================
\section{Cyclotomic quotients\label{sec:cyclotomic}}
%===================================================

In this final section we define cyclotomic quotients of the quantum affine wreath algebras and prove some of their key properties.  These quotients are natural analogues of cyclotomic quotients of affine Hecke algebras (see \cref{eg:affHecke}).

%-----------------------
\subsection{Definitions}
%-----------------------

Identifying $A$ with $A \otimes 1^{\otimes n-1}$, we can naturally view $P_1(A) = A[X_1]$ as a subalgebra of $P_n(A)$.  Let
\[
  f \in Z(A)_0[X_1]
\]
be a monic polynomial of degree $d$ in $X_1$ with coefficients in $Z(A)_0$, the even part of the center of $A$. We write
\begin{equation} \label{festival}
  f = X_1^d + a_{(d-1)} X_1^{d-1} + \dotsb + a_{(1)} X_1 + a_{(0)},
\end{equation}
with $a_{(i)} \in Z(A)_0$.  We assume that $a_{(0)}$ is invertible.

We define the corresponding \emph{quantum cyclotomic wreath algebra} to be
\[
  H_n^f(A,z) := H_n^\aff(A,z)/\langle f \rangle,
\]
where $\langle f \rangle$ denotes the two-sided ideal in $H_n^\aff(A,z)$ generated by $f$.  We call $d$ the \emph{level} of $H_n^f(A,z)$.  Since $f \in H_{n,+}^\aff(A,z) \subseteq H_n^\aff(A,z)$, we can also define
\[
  H_{n,+}^f(A,z) := H_{n,+}^\aff(A,z)/\langle f \rangle_+,
\]
where $\langle f \rangle_+$ denotes the two-sided ideal in $H_{n,+}^\aff(A,z)$ generated by $f$.

Let $f_1:=f$ and, for $2\leq i\leq n$, define
\begin{equation} \label{fi-def}
  f_{i} := T_{i-1} \dotsm T_2T_1 f_1 T_1T_2 \dotsm T_{i-1}.
\end{equation}
It follows immediately that
\begin{equation} \label{drop}
  f_i \text{ commutes with all elements of } A^{\otimes n} \text{ for all } 1 \le i \le n.
\end{equation}

\begin{lem}\label{lem:Xid}
  For $1\leq i\leq n-1$, we have
  \[
    f_i-X_i^d \in H_i(A,z) + \sum_{e=1}^{d-1}\kk[X_1,\dotsc, X_{i-1}]_{\leq d-e} X_{i}^e H_i(A,z).
  \]
  where $\kk[X_1,\dotsc, X_{i-1}]_{\leq d-e}$ denotes the space of polynomials of degree less than or equal to $d-e$.
\end{lem}

\begin{proof}
  We prove this by induction. For $i=1$, the result is immediate. Now assuming the result true for all $1\leq j\leq i$, we have
  \begin{align*}
    f_{i+1}-X_{i+1}^d
    \ &\stackrel{\mathclap{\cref{TiXd}}}{=}\ \, T_i(f_i-X_i^d) T_i - zt_{i,i+1} \sum_{k=1}^{d-1} X_i^k X_{i+1}^{d-k} T_i \\
    &\in T_i H_i(A,z) T_i + T_i \sum_{e=1}^{d-1} \kk[X_1,\dotsc, X_{i-1}]_{\leq d-e}X_i^e H_i(A,z) T_i - z\sum_{k=1}^{d-1}X_i^kX_{i+1}^{d-k}t_{i,i+1}T_i \\
    &\subseteq H_{i+1}(A,z) + \sum_{e=1}^{d-1}\kk[X_1,\dotsc, X_i]_{\leq d-e} X_{i+1}^e H_{i+1}(A,z),
  \end{align*}
  where, for the final inclusion, we used \cref{TiXd}.
\end{proof}

Consider the algebra homomorphism given by the composition
\begin{equation}\label{eq:compmaps}
  \eta \colon H_{n,+}^\aff(A,z) \hookrightarrow H_n^\aff(A,z) \twoheadrightarrow H_n^f(A,z),
\end{equation}
where the first map is the natural inclusion and the second is the projection.

\begin{lem}
  The map $\eta$ is surjective.
\end{lem}

\begin{proof}
  Notice that
  \[
    X_1^{-1}
    = a_{(0)}^{-1} X_1^{-1} f - a_{(0)}^{-1} \left( X_1^{d-1} + a_{(d-1)} X_1^{d-2} + \dotsb + a_{(2)} X_1 + a_{(1)} \right),
  \]
  and so
  \[
    \eta \left( - a_{(0)}^{-1} \left( X_1^{d-1} + a_{(d-1)} X_1^{d-2} + \dotsb + a_{(2)} X_1 + a_{(1)} \right) \right)
    = X_1^{-1} \in H^f_n(A,z).
  \]
  It then follows by induction that $X_{i+1}^{-1} = T_i^{-1} X^{-1}_{i} T_i^{-1} \in \eta\left(H^\aff_{n,+}(A,z)\right)$ for all $1\leq i\leq n-1$, which gives the result.
\end{proof}

\begin{lem} \label{coffee}
  We have
  \[
    H_n(A,z) f H_n(A,z)
    = \sum_{\substack{1 \le i \le n \\ w \in S_n}} A^{\otimes n} f_i T_w
    = \sum_{i=1}^n f_i H_n(A,z).
  \]
\end{lem}

\begin{proof}
  We have
  \begin{align*}
    H_n(A,z) f H_n(A,z)
    &= \sum_{v \in S_n} H_n(A,z) f A^{\otimes n} T_v \\
    &= \sum_{v \in S_n} H_n(A,z) f T_v & \text{(by \cref{drop})} \\
    &= \sum_{i=1}^n \sum_{\substack{x,v \in S_n \\ x(1)=1}} A^{\otimes n} T_{i-1} \dotsm T_1 T_x f T_v \\
    &= \sum_{i=1}^n \sum_{\substack{x,v \in S_n \\ x(1)=1}} A^{\otimes n} T_{i-1} \dotsm T_1 f T_x T_v \\
    &= \sum_{i=1}^n \sum_{v \in S_n} A^{\otimes n} T_{i-1}\dotsm T_1 f T_v \\
    &= \sum_{i=1}^n \sum_{u \in S_n} A^{\otimes n} T_{i-1}\dotsm T_1 f T_1 \dotsm T_{i-1} T_u \\
    &= \sum_{i=1}^n \sum_{u \in S_n} A^{\otimes n} f_i T_u,
  \end{align*}
  where the sixth equality follows from the fact that the $T_j$ are invertible.  The final equality in the statement of the lemma then follows from \cref{drop}.
\end{proof}

\begin{lem} \label{polar}
  For all $1 \le i \le n$, we have
  \[
    \big( H_n(A,z) f H_n(A,z) \big) \cap \big( X_i H_{n,+}^\aff(A,z) \big) = 0.
  \]
\end{lem}

\begin{proof}
  By \cref{coffee}, any element of the intersection is of the form
  \[
    \sum_{j,w} \ba_{j,w} f_j T_w \in X_i H_{n,+}^\aff(A,z)
  \]
  for some $\ba_{j,w} \in A^{\otimes n}$.  It follows from \cref{basis,lem:Xid} that $\ba_{j,w} = 0$ whenever $j \ne i$.  Then, by \cref{TiXd}, the constant term (i.e.\ the term of degree zero in the $X_j$) of $\sum_w \ba_{i,w} f_i T_w$, which must equal zero, is
  \[
    0 = \sum_w \ba_{i,w} T_{i-1} \dotsm T_1 a_{(0)} T_1 \dotsm T_{i-1} T_w
    \stackrel{\cref{TF}}{=} T_{i-1} \dotsm T_1 T_1 \dotsm T_{i-1} \sum_w \ba_{i,w} s_{1,i}(a_{(0)}) T_w.
  \]
  Since the $T_j$ are invertible, as is $a_{(0)}$, it follows from \cref{basis} that $\ba_{i,w} = 0$ for all $w \in S_n$.
\end{proof}

\begin{prop}\label{prop:isoplus}
  The map $\eta$ induces an isomorphism
  \[
    H_{n,+}^f(A,z) \cong H_n^f(A,z).
  \]
\end{prop}

\begin{proof}
  We need to show that $\langle f \rangle_+=\eta^{-1}(\langle f \rangle)$, which is to say that
  \[
    \langle f \rangle\cap H_{n,+}^\aff(A,z)=\langle f \rangle_+.
  \]
  Clearly $\langle f \rangle_+\subseteq \langle f \rangle\cap H_{n,+}^\aff(A,z)$; so we need to show the other inclusion.

  Define a partial order on $\Z^n$ by $\lambda\leq \mu$ if $\lambda_i\leq \mu_i$ for all $1\leq i\leq n$. Using \cref{snowy,TF}, any element of the form
  \[
    \left( \sum_{\substack{\lambda \in \Z^n \\ w \in S_n}} \ba_{\lambda,w} X^{\lambda} T_w \right) f \left( \sum_{\substack{\mu \in \Z^n \\ v \in S_n}} \bb_{\mu,v} X^\mu T_v\right)
    = \sum_{\lambda,\mu,w,v} \ba_{\lambda,w} X^{\lambda}T_w  \bb_{\mu,v}X^\mu fT_v
  \]
  can be written in the form $\sum_{\lambda,w,v} \bc_{\lambda,w,v}X^\lambda T_w f T_v$.  Thus, it suffices to show that if
  \begin{equation}\label{laTfT}
    \sum_{\lambda,w,v} \bc_{\lambda,w,v}X^\lambda T_w f T_v \in H_{n,+}^\aff(A,z),
  \end{equation}
  then $\lambda \geq 0$ whenever $\bc_{\lambda,w,v} \ne 0$ for some $w,v \in S_n$.

  Take a minimal element $\lambda = (\lambda_1,\dotsc,\lambda_n)$ appearing in \cref{laTfT} with $\bc_{\lambda,w,v} \ne 0$.  Suppose, towards a contradiction, that $\lambda_i < 0$ for some $1 \le i \le n$.   Since $f \in A[X_1]$, by \cref{lem:passthrough,polyinv} we have
  \[
    T_w f T_v \in \kk[X_1,\dotsc,X_n] H_n(A,z)
  \]
  for all $w,v$.  Thus, by \cref{basis} and the minimality of $\lambda$, we must  have
  \[
    X^\lambda \sum_{w,v} \bc_{\lambda,w,v} T_w f T_v \in H_{n,+}^\aff(A,z).
  \]
  Since $\lambda_i < 0$, \cref{basis} implies that $\sum_{w,v} \bc_{\lambda,w,v} T_w f T_v \in X_i H_{n,+}^\aff(A,z)$.
  But this is impossible by \cref{polar}.
\end{proof}
%-------------------------
\subsection{Basis Theorem}
%-------------------------

We now prove a basis theorem for $H^f_{n,+}(A,z)$, which also gives a basis theorem for $H^f_n(A,z)$ in light of \cref{prop:isoplus}.  We follow the methods of \cite[\S 7.5]{Kle05} and \cite[\S 6.3]{Sav17}.

For $I = \{i_1 < \dotsb < i_k\}\subseteq \{1,\dotsc,n\}$, let
\[
  f_I = f_{i_1} f_{i_2} \dotsm f_{i_k} \in H_{n,+}^\aff(A,z).
\]
We also define
\begin{gather*}
  \Omega_n := \{(\lambda, I) : I \subseteq \{1,\dotsc,n\},\ \lambda \in \N^n,\ \lambda_i < d \text{ whenever } i \notin I\},
  \\
  \Omega_n^+ := \{(\lambda,I) \in \Omega_n : I \neq \varnothing \}.
\end{gather*}

\begin{lem}\label{lem:rightbasis}
  We have that $H_{n,+}^\aff(A,z)$ is a free right $H_n(A,z)$-module with basis
  \[
    \{ X^\lambda f_I : (\lambda,I) \in \Omega_n\}.
  \]
\end{lem}

\begin{proof}
  Consider the lexicographic ordering $\prec$ on $\N^n$.  Define a function
  \[
    \gamma \colon \Omega_n \to \N^n,\qquad
    \gamma(\lambda,I)
    = (\gamma_1, \dotsc, \gamma_n),\qquad\text{ where }\gamma_i
    =
    \begin{cases}
      \lambda_i & \text{ if }i\not\in I, \\
      \lambda_i+d & \text{ if }i\in I.
    \end{cases}
  \]
  Using induction on $n$ and \cref{lem:Xid}, we see that, for all $(\lambda,I)\in\Omega_n$,
  \begin{equation}\label{eq:gammafI}
    X^\lambda f_I-X^{\gamma(\lambda,I)}\in \sum_{\mu\prec \lambda}X^\mu H_n(A,z).
  \end{equation}
  Now, $\gamma \colon \Omega_n \to \N^n$ is a bijection and, by \cref{currbasis}, $\{X^\lambda : \lambda \in \N^n\}$ is a basis for $H_{n,+}^\aff(A,z)$ as a right $H_n(A,z)$-module.  Thus the lemma follows from \cref{eq:gammafI}.
\end{proof}

\begin{lem}\label{lem:HnfHn}
  We have that $f_n$ commutes with all elements of $H_{n-1}(A,z)$.
\end{lem}

\begin{proof}
  It follows from the definition \cref{fi-def} of $f_n$ and from the relations \cref{drop,braid,commute} that $f_n$ commutes with $T_i$, $1 \leq i \leq n-2$, and $\ba\in A^{\otimes (n-1)}$.
  \details{
    For $\ba \in A^{\otimes (n-1)}$, we have
    \begin{align*}
      \ba f_n
      &= \ba T_{n-1} \dotsm T_1 f T_1 \dotsm T_{n-1} \\
      &= T_{n-1} \dotsm T_1 s_{1,n}(\ba) f T_1 \dotsm T_{n-1} \\
      &\stackrel{\mathclap{\cref{drop}}}{=}\ T_{n-1} \dotsm T_1 f s_{1,n}(\ba) T_1 \dotsm T_{n-1} \\
      &= T_{n-1} \dotsm T_1 f T_1 \dotsm T_{n-1} \ba \\
      &= f_n \ba.
    \end{align*}
    Now suppose $1 \le i \le n-2$.  Then
    \begin{align*}
      T_i f_n
      &= T_i T_{n-1} \dotsm T_1 f T_1 \dotsm T_{n-1} \\
      &= T_{n-1} \dotsm T_{i+2} T_i T_{i+1} T_i T_{i-1} \dotsm T_1 f T_1 \dotsm T_{n-1} \\
      &\stackrel{\mathclap{\cref{braid}}}{=}\  T_{n-1} \dotsm T_{i+2} T_{i+1} T_i T_{i+1} T_{i-1} \dotsm T_1 f T_1 \dotsm T_{n-1} \\
      &= T_{n-1} \dotsm T_{i+2} T_{i+1} T_i T_{i-1} \dotsm T_1 T_{i+1} f T_1 \dotsm T_{n-1} \\
      &\stackrel[\mathclap{\cref{commute}}]{\mathclap{\cref{TF}}}{=}\ T_{n-1} \dotsm T_1 f T_{i+1} T_1 \dotsm T_{n-1} \\
      &= T_{n-1} \dotsm T_1 f T_1 \dotsm T_{i-1} T_{i+1} T_i T_{i+1} T_{i+2} \dotsm T_{n-1} \\
      &\stackrel{\mathclap{\cref{braid}}}{=}\  T_{n-1} \dotsm T_1 f T_1 \dotsm T_{i-1} T_i T_{i+1} T_i T_{i+2} \dotsm T_{n-1} \\
      &= T_{n-1} \dotsm T_1 f T_1 \dotsm T_{i-1} T_i T_{i+1} T_{i+2} \dotsm T_{n-1} T_i \\
      &= f_n T_i.
    \end{align*}
  }
\end{proof}

\begin{lem}\label{lem:fplus}
  We have $\langle f \rangle_+=\sum_{i=1}^n \kk[X_1,\dotsc,X_n]f_i H_n(A,z)$.
\end{lem}

\begin{proof}
  We have
  \begin{align*}
    \langle f \rangle_+
    &= H_{n,+}^\aff (A,z)f \kk[X_1,\dotsc,X_n] H_n(A,z) \\
    &= H_{n,+}^\aff (A,z) \kk[X_1,\dotsc,X_n] f H_n(A,z) \\
    &= H_{n,+}^\aff (A,z)f H_n(A,z) \\
    &= \kk[X_1,\dotsc,X_n] H_n(A,z) f H_n(A,z) \\
    &= \sum_{i=1}^n \sum_{w \in S_n} \kk[X_1,\dotsc,X_n] A^{\otimes n} f_i T_w & \text{(by \cref{coffee})}  \\
    &\stackrel{\mathclap{\cref{drop}}}{=}\ \sum_{i=1}^n \sum_{w \in S_n} \kk[X_1,\dotsc,X_n] f_i A^{\otimes n} T_w \\
    &= \sum_{i=1}^n \kk[X_1,\dotsc,X_n] f_i H_n(A,z). \qedhere
  \end{align*}
\end{proof}

\begin{lem}\label{lem:fplus2}
  For $d>0$, we have $\langle f \rangle _+=\sum_{(\lambda ,I)\in\Omega_n^+}X^\lambda f_I H_n(A,z)$.
\end{lem}

\begin{proof}
  We prove this by induction on $n$. When $n=1$, the statement is obvious.  Now suppose that $n>1$, and define $\langle f\rangle'_+:=H_{n-1,+}^\aff(A,z)fH_{n-1,+}^\aff(A,z)$. By the induction hypothesis we have
  \begin{equation}\label{eq:fprime}
    \langle f\rangle'_+=\sum_{(\lambda',I')\in\Omega_{n-1}^+}X^{\lambda'}f_{I'}H_{n-1}(A,z).
  \end{equation}
  Let $J=\sum_{(\lambda ,I)\in\Omega_n^+}X^\lambda f_I H_n(A,z)$.  Clearly $J\subseteq \langle f \rangle _+$, and so we need to show that $ \langle f \rangle _+\subseteq J$.  By \cref{lem:fplus}, it is enough to show that $X^\lambda f_i H_n(A,z)\subseteq J$ for all $\lambda\in\N^n$ and $1\leq i\leq n$. Consider first the case $i=n$ and write $X^\lambda=X_n^{\lambda_n} X^\mu$, where $\mu = (\lambda_1,\dotsc,\lambda_{n-1}) \in \N^{n-1}$.  Expanding $X^\mu$ in terms of the basis of $H_{n-1,+}^\aff(A,z)$ of \cref{lem:rightbasis}, we see that
  \begin{equation}
    X^\lambda f_n H_n(A,z)
    \subseteq \sum_{(\lambda',I')\in\Omega_{n-1}} X_n^{\lambda_n} X^{\lambda'} f_{I'} H_{n-1} (A,z) f_n H_n(A,z)
    \subseteq J,
  \end{equation}
  where the second inclusion follows from \cref{lem:HnfHn}.

  Now consider $X^\lambda f_i H_n(A,z)$, with $1\leq i<n$. As above, we write $X^\lambda=X_n^{\lambda_n} X^\mu$, where $\mu\in\N^{n-1}$. By the induction hypothesis, we have
  \[
    X^{\lambda}f_iH_n(A,z)
    = X_n^{\lambda_n} X^\mu f_i H_n(A,z)
    \subseteq \sum_{(\lambda',I')\in\Omega_{n-1}^+} X_n^{\lambda_n} X^{\lambda'} f_{I'} H_n(A,z).
  \]
  Now we show by induction on $\lambda_n$ that $X_n^{\lambda_n}X^{\lambda'}f_{I'}H_n(A,z)\in J$ for all $(\lambda',I')\in\Omega_{n-1}^+$. This follows immediately from the definition of $\Omega_{n-1}^+$ when $\lambda_n<d$. If $\lambda_n\geq d$, by \cref{lem:Xid} we have
  \begin{align*}
    X_n^{\lambda_n}X^{\lambda'}f_{I'}H_n(A,z)
    &= X_n^{\lambda_n-d} X^{\lambda'} f_{I'} X_n^d H_n(A,z) \\
    &\in X_n^{\lambda_n-d} X^{\lambda'} f_{I'} f_n H_n(A,z) + \sum_{e=0}^{d-1} X_n^{\lambda_n-d+e} \langle f \rangle_+' H_n(A,z).
  \end{align*}
  By the definition of $J$, we have $X_n^{\lambda_n-d}X^{\lambda'}f_{I'}f_nH_n(A,z)\in J$. Now, by \cref{eq:fprime}, for $0\leq e< d$, we have
  \[
    X_n^{\lambda_n-d+e}\langle f \rangle_+'H_n(A,z)
    \subseteq \sum_{(\lambda',I')\in\Omega_{n-1}^+} X_n^{\lambda_n-d+e} X^{\lambda'} f_{I'} H_n(A,z).
  \]
  Since $0 \leq \lambda_n-d+e < \lambda_n$, each term in the above sum is contained in $J$ by the induction hypothesis, which concludes the proof.
\end{proof}

\begin{theo}[Basis theorem for cyclotomic quotients]\label{theo:cyclobasis}
  The canonical images of the elements
  \[
    \{ X^{\lambda}\ba T_w : \lambda\in \N^n,\ \lambda_i <d \ \forall\ i,\ \ba \in B^{\otimes n},\ w \in S_n\}
  \]
  form a basis of $H_{n,+}^f(A,z)$ and of $H_n^f(A,z)$.
\end{theo}

\begin{proof}
  By \cref{lem:rightbasis,lem:fplus2}, the elements $\{X^\lambda f_I : (\lambda,I) \in \Omega_n^+\}$ form a basis for $\langle f \rangle_+$ as a $H_n(A,z)$-right module.  Thus \cref{lem:rightbasis} implies that
  \[
    \{X^{\lambda} : \lambda\in\N^n,\ \lambda_i<d,\ \forall\ i\}
  \]
  is a basis for a complement to $\langle f\rangle_+$ inside $H_{n,+}^\aff$, viewed as a right $H_n(A,z)$-module.
\end{proof}

\begin{rem}
  In the setting of affine Hecke algebras (see \cref{eg:affHecke}), \cref{theo:cyclobasis} recovers \cite[Th.~3.10]{AK94}.  For affine Yokonuma--Hecke algebras  (see \cref{affYHalg}), it was proved in \cite[Th.~4.4]{CPd16}.
\end{rem}

\begin{cor}
  Every level one quantum cyclotomic wreath algebra is isomorphic to $H_n(A,z)$.
\end{cor}

\begin{proof}
  If $f=X_1-1$, then the map $H_n^\aff(A,z)\twoheadrightarrow H^f_n(A,z)\cong H_n(A,z)$ is exactly the map of \cref{prop:jucysmap}.  In general $f=X_1+a$ with $a \in Z(A)$ even and invertible. So the result follows by applying the automorphism $\zeta_{-a}$ of \cref{rescale}.
\end{proof}

%-------------------------------------
\subsection{Cyclotomic Mackey Theorem}
%-------------------------------------

\Cref{theo:cyclobasis} implies that the subalgebra of $H_{n+1}^f(A,z)$ generated by $X_1,\dotsc, X_n$, $A^{\otimes n}\otimes 1$ and $T_1,\dotsc,T_{n-1}$ is isomorphic to $H_n^f(A,z)$.  Thus we can define induction and restriction functors
\[
  {}^f\Ind_n^{n+1} \colon H_n^f(A,z)\md \to H_{n+1}^f(A,z)\md,\quad
  {}^f\Res_n^{n+1} \colon H_{n+1}^f(A,z)\md\to H_{n}^f(A,z)\md.
\]
Let $\Pi \colon H_n^f(A,z)\md \to H_n^f(A,z)\md$ denote the functor that reverses the parity of the elements of a module.

\begin{prop}\label{mackey}
Recall that $d=\deg f$.
\begin{enumerate}
  \item We have that $H^f_{n+1}(A,z)$ is a free right $H^f_n(A,z)$-module with basis
    \[
      \{ X_j^r a_j T_j \dotsm T_n : 0 \leq r <d,\ a \in B,\ 1 \leq j \leq n+1\}.
    \]

  \item \label{mackey-decomp} We have a decomposition of $(H_n^f(A,z), H_n^f(A,z))$-bimodules
    \[
      H_{n+1}^f(A,z)
      = H_n^f(A,z) T_n H_n^f(A,z) \oplus \bigoplus_{0\leq r <d,\ a \in B} X_{n+1}^r a_{n+1} H_n^f(A,z).
    \]

  \item For $0 \le r \le d$ and homogeneous $a \in A$, we have parity-preserving isomorphisms of $(H_n^f(A,z), H_n^f(A,z))$-bimodules
    \begin{gather*}
      H_n^f(A,z)T_n H_n^f(A,z) \cong H_n^f(A,z)\otimes_{H_{n-1}^f(A,z)} H_n^f(A,z) \quad \text{and} \\
      X^r_{n+1}a_{n+1} H_n^f(A,z) \cong \Pi^{\bar{a}}H_n^f(A,z).
    \end{gather*}
  \end{enumerate}
\end{prop}

\begin{proof}
  The proof is almost identical to the proof of \cite[Lemma 7.6.1]{Kle05} and so will be omitted.
\end{proof}

\begin{theo}[Cyclotomic Mackey Theorem] \label{cycloMackey}
  For all $n\in\N_+$, we have a natural isomorphism of functors
  \[
    {}^f\Res^{n+1}_n\, {}^f\Ind^{n+1}_n\cong \id^{\oplus d\dim(A_0)}\oplus \Pi^{\oplus d\dim(A_1)}\oplus {}^f\Ind^{n}_{n-1}\, {}^f\Res^{n}_{n-1}.
  \]
\end{theo}

\begin{proof}
  This follows from \cref{mackey}.
\end{proof}

\begin{rem}
  \Cref{mackey} is the key ingredient in showing that the quantum Frobenius Heisenberg categories of \cite{BS19} act on categories of modules for quantum cyclotomic wreath algebras.  It corresponds to the inversion relation in the quantum Frobenius Heisenberg categories.
\end{rem}

%---------------------------------------
\subsection{Symmetric algebra structure}
%---------------------------------------

By \cref{theo:cyclobasis}, we can define a $\kk$-linear map
\begin{equation} \label{trace}
  \tr_f^n \colon H_n^f(A,z) \to \kk,\quad
  X^\lambda \ba T_w \mapsto \delta_{\lambda,0} \delta_{w,1} \tr(\ba),
\end{equation}
where $\tr(\ba) = \tr^{\otimes n}(\ba)$ is the natural trace map on the tensor product algebra $A^{\otimes n}$ (here, on the right-hand side, $\tr$ is the trace map on $A$).

\begin{theo} \label{chasm}
  The cyclotomic quotient $H_n^f(A,z)$ is a symmetric algebra with trace map $\tr_f^n$.
\end{theo}

\begin{proof}
  Consider the total order on $\N^n$ given by $\lambda < \mu$ if and only if
  \[
    \lambda_n = \mu_n, \dotsc, \lambda_{i+1} = \mu_{i+1} \text{ and } \lambda_i < \mu_i
    \quad \text{for some } 1 \le i \le n.
  \]
  For the remainder of this proof,
  \begin{itemize}
    \item $\lambda$ and $\mu$ will denote elements of $\N^n$ such that $\lambda_i,\mu_i < d$ for all $i$,
    \item $\ba$ and $\bb$ will denote elements of $B^{\otimes n}$, and
    \item $u,v,w$ will denote elements of $S_n$.
  \end{itemize}

  We must verify that the basis given in \cref{theo:cyclobasis} has a left dual basis with respect to $\tr_f^n$.  By \cref{twf}, we have
  \begin{equation} \label{chicken}
    \tr_f^n ( X^{-\lambda} \ba^\vee T_{w^{-1}} X^\mu \bb T_v )
    = \tr_f^n \left( X^{w^{-1}(\mu) - \lambda} \ba^\vee w^{-1}(\bb) T_{w^{-1}} T_v + \sum_{u < w^{-1}} f_u T_u T_v \right)
  \end{equation}
  for some $f_u \in P_n(A)$.

  By \cref{affbasis}, the equation \cref{meatballs} holds in $H_n^\aff(A,z)$.  It follows that
  \[
    T_{w^{-1}} T_w \in T_1 + \sum_{v \ne 1} P_n(A) T_v
    \quad \text{and} \quad
    T_{w^{-1}} T_u \in \sum_{v \ne 1} P_n(A) T_v
    \quad \text{if } w < u.
  \]
  The second equation above also implies that
  \[
    T_u T_v \in \sum_{v' \ne 1} P_n(A) T_{v'}
    \quad \text{whenever } u < v^{-1},
  \]
  since $u < v^{-1} \implies u^{-1} < v$.  Thus, it follows from \cref{chicken,lem:Xid} that $\tr_f^n ( X^{-\lambda} \ba^\vee T_{w^{-1}} X^\lambda \ba T_w ) = 1$ and that
  \[
    \tr_f^n ( X^{-\lambda} \ba^\vee T_{w^{-1}} X^\mu \bb T_v ) = 0
  \]
  whenever
  \begin{itemize}
    \item $w < v$, or
    \item $w=v$ and $\lambda < \mu$, or
    \item $w=v$, $\lambda = \mu$, and $\ba \ne \bb$.
  \end{itemize}
  Thus we can find a left dual basis to the basis given in \cref{theo:cyclobasis} by inverting a unitriangular matrix.

  It remains to prove that the trace map $\tr_f^n$ is symmetric.  Let $\psi$ denote the Nakayama automorphism corresponding to $\tr_f^n$ (see \cref{Nakayama}).  So we want to show that $\psi$ is the identity automorphism.  It follows from \cref{TF,pna} that
  \[
    \tr_f^n(\bb X^\lambda \ba T_w)
    = \delta_{\lambda,0} \delta_{w,1} \tr(\bb \ba)
    = (-1)^{\bar{\ba} \bar{\bb}} \delta_{\lambda,0} \delta_{w,1} \tr(\ba \bb)
    = (-1)^{\bar{\ba} \bar{\bb}} \tr_f^n(X^\lambda \ba T_w \bb).
  \]
  So $\psi(\bb) = \bb$.

  If $\lambda \ne 0$, we have (noting that $\Delta_i$ preserves polynomial degree)
  \[
    \tr_f^n(T_i X^\lambda \ba T_w)
    \stackrel{\cref{demazure}}{=} \tr_f^n \left( X^{s_i(\lambda)} s_i(\ba) T_i T_w + z t_{i,i+1} \ba \Delta_i(X^\lambda) T_w \right)
    = 0
    = \tr_f^n(X^\lambda \ba T_w T_i).
  \]
  We also have
  \begin{align*}
    \tr_f^n(T_i \ba T_w)
    &= \tr_f^n(s_i(\ba) T_i T_w) \\
    &\stackrel{\mathclap{\cref{meatballs}}}{=}\
    \begin{cases}
      \tr_f^n ( s_i(\ba) T_{s_i w} ) & \text{if } \ell(s_i w) > \ell(w), \\
      \tr_f^n \left( s_i(\ba) T_{s_i w} + z t_{i,i+1} T_w \right) & \text{if } \ell(s_i w) < \ell(w)
    \end{cases} \\
    &= \begin{cases}
      0 & \text{if } w \ne s_i, \\
      \tr(s_i(\ba)) = \tr(\ba) & \text{if } w = s_i.
    \end{cases}
  \end{align*}
  Similarly,
  \[
    \tr_f^n(\ba T_w T_i) =
    \begin{cases}
      0 & \text{if } w \ne s_i, \\
      \tr(\ba) & \text{if } w = s_i.
    \end{cases}
  \]
  Thus $\psi(T_i) = T_i$.

  Now, if $\psi(X_i) = X_i$ for some $1 \le i \le n-1$, then
  \[
    \psi(X_{i+1})
    \stackrel{\cref{inverse}}{=} \psi(T_i X_i T_i)
    = \psi(T_i) \psi(X_i) \psi(T_i)
    = T_i X_i T_i
    \stackrel{\cref{inverse}}{=} X_{i+1}.
  \]
  Therefore, it remains to show that $\psi(X_1) = X_1$.  That is, we need to show
  \begin{equation} \label{cold}
    \tr_f^n(X_1 X^\lambda \ba T_w) = \tr_f^n(X^\lambda \ba T_w X_1)
    \quad \text{for all } \lambda, \ba, w.
  \end{equation}
  It follows from \cref{commute,pna} that \cref{cold} holds when $w(1)=1$.

  Now suppose $w = s_1 v$ for some $v \in S_n$ with $v(1)=1$.  Then
  \begin{equation} \label{pizza}
    \tr_f^n(X^\lambda \ba T_w X_1)
    = \tr_f^n(X^\lambda \ba T_1 T_v X_1)
    \stackrel{\cref{commute}}{=}\ \tr_f^n(X^\lambda \ba T_1 X_1 T_v)
    \stackrel[\cref{TF}]{\cref{inverse}}{=}\ \tr_f^n(X^\lambda X_2 T_1^{-1} s_1(\ba) T_v).
  \end{equation}
  It is clear that \cref{pizza} is equal to zero unless $\lambda_2 = d-1$ and $\lambda_3 = \dotsb = \lambda_n = 0$, which we assume from now on.   Now, for $m > 1$, we have
  \begin{equation} \label{camp}
    T_1 X_1^m
    \stackrel{\cref{demazure}}{=} X_2^m T_1 - z t_{1,2} \sum_{r=1}^m X_1^{m-r} X_2^r
    \stackrel{\cref{skein}}{=} X_2^m T_1^{-1} - z t_{1,2} \sum_{r=1}^{m-1} X_1^{m-r} X_2^r.
  \end{equation}
  Using \cref{festival}, this gives
  \[
    X_2^dT_1^{-1}
    = - \sum_{m=0}^{d-1} T_1 a_{(m)} X_1^m + z t_{1,2} \sum_{r=1}^{d-1} X_1^{d-r} X_2^r.
  \]
  Therefore, from \cref{pizza,trace}, we have
  \begin{equation} \label{scarf}
    \tr_f^n(X^\lambda \ba T_w X_1)
    = - z \sum_{m=0}^{d-1} \tr_f^n ( X_1^{\lambda_1} T_1 a_{(m)} X_1^m s_1(\ba) T_v )
    \stackrel{\cref{camp}}{=} - z \sum_{m=0}^{d-1} \tr_f^n ( X_1^{\lambda_1} s_1(a_{(m)}) X_2^m \ba T_1 T_v)
    = 0.
  \end{equation}

  Now we consider the general case where $w(1) \ne 1$.  Then we can write a reduced expression $w = w_1 s_1 w_2$, where $w_1(1) = w_2(1) = 1$.  Then, for all $g \in P_n(A)$, we have
  \begin{align}
    \tr_f^n(g T_w X_1)
    &= \tr_f^n(g T_{w_1} T_1 T_{w_2} X_1) \nonumber \\
    &\stackrel{\mathclap{\cref{twf}}}{=}\ \tr_f^n ( T_{w_1} w^{-1}(g) T_1 T_{w_2} X_1 ) + \sum_{u < w_1} \tr_f^n ( T_u f_u T_1 T_{w_2} X_1 ) \nonumber \\
    &= \tr_f^n ( w^{-1}(g) T_1 T_{w_2} X_1 T_{w_1}) + \sum_{u < w_1} \tr_f^n ( f_u T_1 T_{w_2} X_1 T_u)
    \nonumber \\
    &\stackrel{\mathclap{\cref{commute}}}{=}\ \tr_f^n ( w^{-1}(g) T_1 T_{w_2} T_{w_1} X_1) + \sum_{u < w_1} \tr_f^n ( f_u T_1 T_{w_2} T_u X_1) \quad \text{for some } f_u \in P_n(A), \label{shred}
  \end{align}
  where, in the third equality, we used the fact that $\psi(T_v) = T_v$ for all $v \in S_n$ and, in the fourth equality, we have used the fact that $u(1)=1$ whenever $u < w_1$.  Now, since $w_2(1) = 1$ and $u(1) = 1$ for all $u \le w_1$, we have that
  \[
    T_{w_2} T_u \in \sum_{v : v(1)=1} P_n(A) T_v
    \quad \text{for all } u \le w_1.
  \]
  Thus, it follows from \cref{scarf,TF} that \cref{shred} is equal to zero.  Since $\tr_f^n(X_1 g T_{w_1} T_1 T_{w_2}) = 0$, we are done.
\end{proof}

\begin{rem}
  For affine Hecke algebras (see \cref{eg:affHecke}), \cref{chasm} was proved in \cite[Th.~5.1]{MM98}.  For affine Yokonuma--Hecke algebras (see \cref{affYHalg}), it was proved in \cite[Th.~7.1]{CPd16}.
\end{rem}

%-----------------------------------------
\subsection{Frobenius extension structure}
%-----------------------------------------

Let
\[
  A_{n+1} := 1^{\otimes n} \otimes A = \Span_\kk \{a_{n+1} : a \in A\} \subseteq A^{\otimes (n+1)}.
\]
By \cref{mackey}\cref{mackey-decomp}, we have a decomposition of $(H_n^f(A,z), H_n^f(A,z))$-bimodules
\begin{equation} \label{milk}
  H_{n+1}^f(A,z) = A_{n+1} H_n^f(A,z) \oplus \bigoplus_{r=1}^{d-1} X_{n+1}^r A_{n+1} H_n^f(A,z) \oplus H_n^f(A,z) T_n H_n^f(A,z).
\end{equation}
Define the \emph{partial trace map}
\[
  \tr_{n+1}^f \colon H_{n+1}^f(A,z) \to H_n^f(A,z)
\]
to be the homomorphism of $(H_n^f(A,z), H_n^f(A,z))$-bimodules given by the projection onto the first summand in \cref{milk} followed by the map
\[
  A_{n+1} H_n^f(A,z) \to H_n^f(A,z),\quad a_{n+1} \alpha \mapsto \tr_A(a) \alpha,\quad a \in A,\ \alpha \in H_n^f(A,z).
\]
It follows that
\[
  \tr_f^{n+1} = \tr_1^f \circ \tr_2^f \circ \dotsb \circ \tr_{n+1}^f = \tr_f^n \circ \tr_{n+1}^f.
\]

\begin{prop} \label{sugar}
  The quantum cyclotomic wreath algebra $H_{n+1}^f(A,z)$ is a Frobenius extension of $H_n^f(A,z)$ with trace map $\tr_{n+1}^f$.
\end{prop}

\begin{proof}
  Since $H_n^f(A,z) \subseteq H_{n+1}^f(A,z)$ are both symmetric algebras, it follows from \cite[Cor.~7.4]{PS16} that $H_{n+1}^f(A,z)$ is a Frobenius extension of $H_n^f(A,z)$ with trace map
  \[
    \alpha \mapsto \sum_{\beta \in Y} \tr_f^{n+1}(\beta^\vee \alpha) \beta,
  \]
  where $Y$ is a basis of $H_n^f(A,z)$.  (Note that $\beta^\vee$ denotes the \emph{right} dual of $\beta$ in \cite{PS16}, whereas it denotes the \emph{left} dual in the current paper.)  Since
  \[
    \sum_{\beta \in Y} \tr_f^{n+1}(\beta^\vee \alpha) \beta
    = \sum_{\beta \in Y} \tr_f^n \left( \tr_{n+1}^f(\beta^\vee \alpha) \right) \beta
    = \sum_{\beta \in Y} \tr_f^n \left( \beta^\vee \tr_{n+1}^f(\alpha) \right) \beta
    = \tr_{n+1}^f(\alpha),
  \]
  the result follows.
\end{proof}

It follows from \cref{sugar} that the functors ${}^f \Ind_n^{n+1}$ and ${}^f \Res_n^{n+1}$ are both left and right adjoint to each other.  Indeed, induction is always left adjoint to restriction.  It is also right adjoint to restriction precisely when the larger algebra is a Frobenius extension of the smaller.

%===========
% References
%===========

\bibliographystyle{alphaurl}
\bibliography{QAWA}

\def\cprime{$'$} \newcommand{\arxiv}[1]{\href{http://arxiv.org/abs/#1}{\tt
  arXiv:\nolinkurl{#1}}}
\begin{thebibliography}{BSW18}

\bibitem[AK94]{AK94}
S.~Ariki and K.~Koike.
\newblock A {H}ecke algebra of
  {$({\mathbf{Z}}/r{\mathbf{Z}})\wr{\mathfrak{S}}_n$} and construction of its
  irreducible representations.
\newblock {\em Adv. Math.}, 106(2):216--243, 1994.
\newblock \href {https://doi.org/10.1006/aima.1994.1057}
  {\path{doi:10.1006/aima.1994.1057}}.

\bibitem[BSW]{BS19}
J.~Brundan, A.~Savage, and B.~Webster.
\newblock Quantum {F}robenius {H}eisenberg categorification.
\newblock In preparation.

\bibitem[BSW18]{BSW18}
J.~Brundan, A.~Savage, and B.~Webster.
\newblock On the definition of quantum {H}eisenberg category.
\newblock {\em Algebra Number Theory}, 2018.
\newblock To appear. \arxiv{1812.04779}.

\bibitem[CL12]{CL12}
S.~Cautis and A.~Licata.
\newblock Heisenberg categorification and {H}ilbert schemes.
\newblock {\em Duke Math. J.}, 161(13):2469--2547, 2012.
\newblock \arxiv{1009.5147}.
\newblock \href {https://doi.org/10.1215/00127094-1812726}
  {\path{doi:10.1215/00127094-1812726}}.

\bibitem[Cou16]{Cou16}
C.~Couture.
\newblock Skew-zigzag algebras.
\newblock {\em SIGMA Symmetry Integrability Geom. Methods Appl.}, 12:Paper No.
  062, 19, 2016.
\newblock \href {https://doi.org/10.3842/SIGMA.2016.062}
  {\path{doi:10.3842/SIGMA.2016.062}}.

\bibitem[CPd14]{CPd14}
M.~Chlouveraki and L.~Poulain~d'Andecy.
\newblock Representation theory of the {Y}okonuma-{H}ecke algebra.
\newblock {\em Adv. Math.}, 259:134--172, 2014.
\newblock \arxiv{1302.6225}.
\newblock \href {https://doi.org/10.1016/j.aim.2014.03.017}
  {\path{doi:10.1016/j.aim.2014.03.017}}.

\bibitem[CPd16]{CPd16}
M.~Chlouveraki and L.~Poulain~d'Andecy.
\newblock Markov traces on affine and cyclotomic {Y}okonuma-{H}ecke algebras.
\newblock {\em Int. Math. Res. Not. IMRN}, (14):4167--4228, 2016.
\newblock \arxiv{1406.3207}.
\newblock \href {https://doi.org/10.1093/imrn/rnv257}
  {\path{doi:10.1093/imrn/rnv257}}.

\bibitem[CW]{CW15}
W.~Cui and J.~Wan.
\newblock Modular representations and branching rules for affine and cyclotomic
  {Y}okonuma-{H}ecke algebras.
\newblock \arxiv{1506.06570}.

\bibitem[Dem74]{Dem74}
M.~Demazure.
\newblock D\'{e}singularisation des vari\'{e}t\'{e}s de {S}chubert
  g\'{e}n\'{e}ralis\'{e}es.
\newblock {\em Ann. Sci. \'{E}cole Norm. Sup. (4)}, 7:53--88, 1974.
\newblock Collection of articles dedicated to Henri Cartan on the occasion of
  his 70th birthday, I.
\newblock URL: \url{http://www.numdam.org/item?id=ASENS_1974_4_7_1_53_0}.

\bibitem[DJ86]{DJ86}
R.~Dipper and G.~James.
\newblock Representations of {H}ecke algebras of general linear groups.
\newblock {\em Proc. London Math. Soc. (3)}, 52(1):20--52, 1986.
\newblock \href {https://doi.org/10.1112/plms/s3-52.1.20}
  {\path{doi:10.1112/plms/s3-52.1.20}}.

\bibitem[HK01]{HK01}
R.~S. Huerfano and M.~Khovanov.
\newblock A category for the adjoint representation.
\newblock {\em J. Algebra}, 246(2):514--542, 2001.
\newblock \arxiv{math/0002060}.
\newblock \href {https://doi.org/10.1006/jabr.2001.8962}
  {\path{doi:10.1006/jabr.2001.8962}}.

\bibitem[Kle05]{Kle05}
A.~Kleshchev.
\newblock {\em Linear and projective representations of symmetric groups},
  volume 163 of {\em Cambridge Tracts in Mathematics}.
\newblock Cambridge University Press, Cambridge, 2005.
\newblock \href {https://doi.org/10.1017/CBO9780511542800}
  {\path{doi:10.1017/CBO9780511542800}}.

\bibitem[KM19]{KM15}
A.~Kleshchev and R.~Muth.
\newblock Affine zigzag algebras and imaginary strata for {KLR} algebras.
\newblock {\em Trans. Amer. Math. Soc.}, 371(7):4535--4583, 2019.
\newblock \arxiv{1511.05905}.
\newblock \href {https://doi.org/10.1090/tran/7464}
  {\path{doi:10.1090/tran/7464}}.

\bibitem[LS13]{LS13}
A.~Licata and A.~Savage.
\newblock Hecke algebras, finite general linear groups, and {H}eisenberg
  categorification.
\newblock {\em Quantum Topol.}, 4(2):125--185, 2013.
\newblock \arxiv{1101.0420}.
\newblock \href {https://doi.org/10.4171/QT/37} {\path{doi:10.4171/QT/37}}.

\bibitem[Lus89]{Lus89}
G.~Lusztig.
\newblock Affine {H}ecke algebras and their graded version.
\newblock {\em J. Amer. Math. Soc.}, 2(3):599--635, 1989.
\newblock \href {https://doi.org/10.2307/1990945} {\path{doi:10.2307/1990945}}.

\bibitem[MM98]{MM98}
G.~Malle and A.~Mathas.
\newblock Symmetric cyclotomic {H}ecke algebras.
\newblock {\em J. Algebra}, 205(1):275--293, 1998.
\newblock \href {https://doi.org/10.1006/jabr.1997.7339}
  {\path{doi:10.1006/jabr.1997.7339}}.

\bibitem[PS16]{PS16}
J.~Pike and A.~Savage.
\newblock Twisted {F}robenius extensions of graded superrings.
\newblock {\em Algebr. Represent. Theory}, 19(1):113--133, 2016.
\newblock \arxiv{1502.00590}.
\newblock \href {https://doi.org/10.1007/s10468-015-9565-4}
  {\path{doi:10.1007/s10468-015-9565-4}}.

\bibitem[RS17]{RS17}
D.~Rosso and A.~Savage.
\newblock A general approach to {H}eisenberg categorification via wreath
  product algebras.
\newblock {\em Math. Z.}, 286(1-2):603--655, 2017.
\newblock \arxiv{1507.06298}.
\newblock \href {https://doi.org/10.1007/s00209-016-1776-9}
  {\path{doi:10.1007/s00209-016-1776-9}}.

\bibitem[Sav19]{Sav18}
A.~Savage.
\newblock Frobenius {H}eisenberg categorification.
\newblock {\em Algebr. Comb.}, 2(5):937--967, 2019.
\newblock \href {http://arxiv.org/abs/1802.01626} {\path{arXiv:1802.01626}},
  \href {https://doi.org/10.5802/alco.73} {\path{doi:10.5802/alco.73}}.

\bibitem[Sav20]{Sav17}
A.~Savage.
\newblock Affine wreath product algebras.
\newblock {\em Int. Math. Res. Not. IMRN}, (10):2977--3041, 2020.
\newblock \href {http://arxiv.org/abs/1709.02998} {\path{arXiv:1709.02998}},
  \href {https://doi.org/10.1093/imrn/rny092} {\path{doi:10.1093/imrn/rny092}}.

\bibitem[WW08]{WW08}
J.~Wan and W.~Wang.
\newblock Modular representations and branching rules for wreath {H}ecke
  algebras.
\newblock {\em Int. Math. Res. Not. IMRN}, rnn128, 2008.
\newblock \arxiv{0806.0196}.
\newblock \href {https://doi.org/10.1093/imrn/rnn128}
  {\path{doi:10.1093/imrn/rnn128}}.

\end{thebibliography}

\end{document}